\documentclass{elsarticle}
\DeclareFontFamily{OML}{script}{}
\DeclareFontShape{OML}{script}{m}{it}
{ <5-20> rsfs10 }{}
\DeclareMathAlphabet{\mathscript}{OML}{script}{m}{it}

\renewcommand{\mathcal}[1]{{\mathscript #1}\hspace{0.2ex}}
\usepackage{color}
\ifx\red\undefined
\newcommand{\red}{\color{red}\huge}

\fi
\usepackage[ansinew]{inputenc}
\usepackage[encapsulated]{CJK}

\usepackage{graphicx,graphics}
\usepackage{pifont}
\usepackage{amsthm}

\usepackage{float}
\usepackage{amscd}
\usepackage{amsmath}
\usepackage{latexsym}
\usepackage{amsfonts}
\usepackage{amssymb}
\usepackage{color}
\usepackage{multicol}
\usepackage{bm}
\allowdisplaybreaks[4]

\newcommand{\re}[1]{\mbox{\rm$($\ref{#1}$)$}}
\newcommand{\dis}{\displaystyle}
\newcommand{\m}{\hspace{1em}}
\newcommand{\mm}{\hspace{2em}}
\newcommand{\p}{\partial}
\newcommand{\x}{\vspace*{1ex}}

\newcommand{\Rmnum}[1]{\uppercase\expandafter{\romannumeral #1}}

\renewcommand{\epsilon}{\varepsilon}
\renewcommand{\t}{\widetilde}

\ifx\text\undefined
\newcommand{\text}{\mbox}
\fi
\ifx\operatorname\undefined
\newcommand{\operatorname}{\mathop}
\fi

\newcommand\be{\begin{equation}}
\newcommand\ee{\end{equation}}
\newcommand\bea{\begin{eqnarray}}
\newcommand\eea{\end{eqnarray}}
\newcommand\beaa{\begin{eqnarray*}}
\newcommand\eeaa{\end{eqnarray*}}
\newcommand{\dif}{\mathrm{d}}

\newenvironment{eqa}{\begin{equation}%
  \begin{array}{rcl}}{\end{array}\end{equation}}

\newcommand\beqa{\begin{eqa}}
\newcommand\eeqa{\end{eqa}}

\numberwithin{equation}{section}
\renewcommand{\theequation}{\arabic{section}.\arabic{equation}}
\renewcommand{\tilde}{\widetilde}
\renewcommand{\hat}{\widehat}
\renewcommand{\bar}{\overline}

\newtheorem{thm}{Theorem}[section]

\newtheorem{lem}[thm]{Lemma}

\newtheorem{rem}{Remark}[section]

\newcommand{\void}[1]{}

\def\theequation{\arabic{section}.\arabic{equation}}

\numberwithin{equation}{section}

\begin{document}
\title{Bifurcation for a free boundary problem modeling a small arterial plaque\footnote{\today}
}
\author{Xinyue Evelyn Zhao}\author{Bei Hu}
\address{Department of Applied and Computational Mathematics and Statistics, \\
University of Notre Dame,  Notre Dame, IN 46556, USA, \\ xzhao6@nd.edu, b1hu@nd.edu }

\begin{abstract}
Atherosclerosis, hardening of the arteries, originates from small plaque in the arteries; it is a major cause of disability and premature death in the United States and worldwide. In this paper, we study the bifurcation of a highly nonlinear and highly coupled PDE model describing the growth of arterial plaque in the early stage of atherosclerosis. The model involves LDL and HDL cholesterols, macrophage cells, and foam cells, with the interface separating the plaque and blood flow regions being a free boundary. We establish finite branches of symmetry-breaking stationary solutions which bifurcate from the radially symmetric solution. Since plaque in reality is unlikely to be strictly radially symmetric, our result would be useful to explain the asymmetric shapes of plaque.
\end{abstract}

\maketitle

\section{Introduction}
Atherosclerosis, known as an inflammatory disease, is a major cause of disability and premature death in the United States and worldwide. It occurs when fat, cholestrol, and other substances build up in and on the artery walls. These deposits are called plaques, which harden and narrow the arteries over time. The plaque can rupture, triggering a blood clot which restricts blood flow. During this process, a heart attack, stroke, or sudden cardiac death may occur.  Every year about 735,000 Americans have a heart attack, and about 610,000 people die of heart diseases in the United States --- that is 1 in every 4 deaths (cf.,\cite{web4,web3}).

There are several mathematical models that describe the growth of plaque in the arteries (see \cite{CEMR, CMT,FHplaque1,FHHplaque, FHplaque2, mckay2005towards, mukherjee2019reaction}). All of these models recognize the critical role of the ``bad'' cholesterols, low density lipoprotein (LDL), and the ``good'' cholesterols, high density lipoprotein (HDL), in determining whether plaque will grow or shrink. Recently, a free boundary PDE model was proposed in \cite{FHplaque2}, in which a risk-map was generated for any pair values of (LDL, HDL), showing the important influence of LDL and HDL on plaque formation. Later, on the foundation of the model, Hao and Friedman added the impact of reverse cholesterol transport (RCT) in \cite{FHplaque1}. In addition, the existence of a small radially symmetric stationary plaque and its stability condition were theoretically established for a simplified free boundary model in \cite{FHHplaque}. Nevertheless, there is no theoretical work to analyze the bifurcation of plaque model. As the plaque in reality is unlikely to be radially symmetric, it is necessary to investigate the non-radially symmetric solutions. Hence in this paper we shall carry out the bifurcation for the plaque model proposed in  \cite{FHHplaque} (also see \cite[Chapters 7 and 8]{mathbio}).

\begin{figure}[H]
\centering
\includegraphics[height=1.45in]{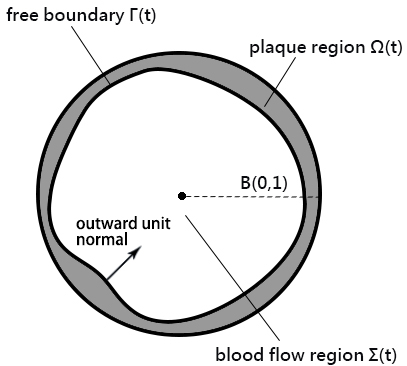}
\caption{\small The cross section of an artery.}
\label{plaque}
\end{figure}

The process of plaque formation is as follows: when a lesion develops in the inner surface of the arterial wall, it enables LDL and HDL to move into the intima and become oxidized by free radicals. Oxidized LDL triggers endothelial cells to secrete chemoattractant proteins that attract macrophage cells (M) from the blood. Macrophage cells
can engulf oxidized LDL, they then become foam cells (F), and the accumulation of foam cells results in the formation of plaque. The effect of oxidized LDL on plaque growth can be reduced by the good cholesterols, HDL: HDL can remove harmful bad cholesterol out from the foam cells and revert foam cells back into macrophage cells; moreover, HDL also competes with LDL on free radicals, decreasing the amount of radicals that are available to oxidize LDL. In the model, we let
\begin{equation*}
\begin{split}
&L\text{ = concentration of LDL,}\hspace{5em} H\text{ = concentration of HDL,}\\
&M\text{ = density of macrophage cells,}\hspace{2em} F\text{ = density of foam cells.}
\end{split}
\end{equation*}
Assuming the artery is a very long circular cylinder with radius 1 (after normalization), we consider a circular cross section of the artery. As can be seen in Fig. \ref{plaque}, the cross section is divided into two regions: blood flow region $\Sigma(t)$ and plaque region $\Omega(t)$, with a moving boundary $\Gamma(t)$ separating these two regions (since plaque can either grow or shrink). The variables $L,H,M,F$ satisfy the following equations in the plaque region $\{\Omega(t), t>0\}$ (cf., \cite[Chapters 7 and 8]{mathbio} and \cite{FHHplaque}):
\begin{gather}
\frac{\partial L}{\partial t} - \Delta L = - k_1 \frac{ML}{K_1 + L} - \rho_1 L,\label{p1}\\
\frac{\partial H}{\partial t} - \Delta H =- k_2 \frac{HF}{K_2 + F} - \rho_2 H,\label{p2}\\
\frac{\partial M}{\partial t} - D\Delta M + \nabla\cdot(M \vec{v}) = -k_1\frac{ML}{K_1+L} + k_2\frac{HF}{K_2+F} + \lambda\frac{ML}{\gamma+H}-\rho_3 M,\label{p3}\\
\frac{\partial F}{\partial t} - D\Delta F + \nabla\cdot(F \vec{v}) = k_1\frac{ML}{K_1+L}-k_2\frac{HF}{K_2+F}-\rho_4 F,\label{p4}
\end{gather}
where $\rho_1$, $\rho_2$, $\rho_3$ and $\rho_4$ denote the natural death rate of $L$, $H$, $M$, and $F$, respectively. In equations \re{p1} --- \re{p4}, the aforementioned transitions between macrophage cells ($M$) and foam cells ($F$) are included: $k_1\frac{ML}{K_1 + L}$ accounts for the fact that $M$ becomes foam cell by combining with $L$, $k_2\frac{HF}{K_2+F}$ describes the removal of foam cell by $H$, and the extra term $\lambda\frac{ML}{\gamma+H}$ in equation \re{p3} models the  effects that oxidized $L$ attracts $M$ while  $H$
decreases this impact by competing for free radicals.

We assume that the density of cells in the plaque is approximately a constant, and take
\begin{equation}
    M+F \equiv M_0\hspace{2em}\text{in } \Omega(t).\label{p5}
\end{equation}
Since there are cells migrating into and out of the plaque, the total number of cells keeps changing and, under the assumption \re{p5}, cells are continuously ``pushing'' each other. This gives rise to an internal pressure among the cells which is associated with the velocity $\vec{v}$ in \re{p3} and \re{p4}. 
We further assume that the plaque  texture is of a porous medium type, and invoke Darcy's law
\begin{equation}
    \vec{v} = -\nabla p \mm (\mbox{the proportional constant is normalized to 1}),\label{p6}
\end{equation}
where $p$ is the internal pressure {\em relative to the outside pressure} (and therefore can admit positive or negative sign). Combining \re{p3} -- \re{p6}, we derive
\begin{equation}
    -\Delta p = \frac{1}{M_0}\Big[\lambda\frac{(M_0-F)L}{\gamma+H}-\rho_3(M_0-F) - \rho_4 F\Big].\label{peqn}
\end{equation}
Due to the assumption \re{p5}, we can decrease the number of equations by 1, and replace $M$ by $M_0 - F$ in \re{p1} -- \re{p4}, hence we shall have 4 PDEs, for $L$, $H$, $F$ and $p$, respectively. In particular, combining with \re{peqn}, we write the equation for $F$ in the following form
\begin{equation}
    \label{Feqn}
    \frac{\p F}{\p t}-D\Delta F - \nabla F\cdot \nabla p = k_1\frac{(M_0-F)L}{K_1+L}-k_2\frac{HF}{K_2+F}-\lambda\frac{F(M_0-F)L}{M_0(\gamma+H)}+(\rho_3-\rho_4)\frac{(M_0-F)F}{M_0}.
\end{equation}

We now turn to the boundary conditions. We assume no flux condition on the blood vessel wall ($r=1$) for all variables (no exchange through the blood vessel): 
\begin{equation}
\frac{\partial L}{\partial r} = \frac{\partial H}{\partial r} = \frac{\partial F}{\partial r} =  \frac{\partial p}{\partial r} = 0 \hspace{2em}\text{at } r=1;\label{p7}
\end{equation}
while on the free boundary $\Gamma(t)$, we take
\begin{eqnarray}
&\frac{\partial L}{\partial {\bf n}} + \beta_1 (L-L_0) = 0\hspace{2em} &\text{on }\Gamma(t),\label{p8}\\
&\frac{\partial H}{\partial {\bf n}}+\beta_1 (H-H_0) = 0\hspace{2em}&\text{on }\Gamma(t),\label{p9}\\
&\frac{\partial F}{\partial {\bf n}}+\beta_2 F = 0\hspace{2em}&\text{on }\Gamma(t),\label{p11}\\
&p = \kappa \hspace{2em} &\text{on }\Gamma(t),\label{p12}
\end{eqnarray}
where ${\bf n}$ is the outward unit normal for $\Gamma(t)$ which points inward towards the blood region (as shown in Fig. \ref{plaque}), and $\kappa$ is the corresponding mean curvature in the direction of ${\bf n}$ (i.e., $\kappa=-\frac{1}{R(t)}$ if $\Gamma(t)=\{r=R(t)\}$). The cell-to-cell adhesiveness constant in front of $\kappa$ is normalized to 1. The flux boundary conditions \re{p8} and \re{p9} are based on the fact that the concentrations of $L$ and $H$ in the blood are $L_0$ and $H_0$, respectively; and the meaning of \re{p11} is similar: there are, of course, no foam cells in the blood.

Furthermore, we assume that the velocity is continuous up to the boundary, so that the free boundary $\Gamma(t)$ moves in the outward normal direction ${\bf n}$ with velocity $\vec{v}$; based on \re{p6}, the normal velocity of the free boundary is defined by
\begin{equation}
    \label{p13}
    V_n = -\frac{\partial p}{\partial {\bf n}}\hspace{2em}\text{on }\Gamma(t).
\end{equation}

In \cite{FHHplaque}, Friedman et al. analyzed the system \re{p1} -- \re{p13} in the radially symmetric case and established the existence of a unique radially symmetric steady state solution in a ring-region $1-\epsilon < r <1$ with $\epsilon$ being small. It is, however, unreasonable to assume plaque is of strictly radially symmetric shape, hence we'd like to investigate the symmetric-breaking bifurcation for the system. To do that, we study the corresponding stationary problem of \re{p1} -- \re{p13}:
\begin{eqnarray}
&- \Delta L = - k_1 \frac{(M_0-F)L}{K_1 + L} - \rho_1 L\hspace{1em} &\text{in } \Omega,\label{e1}\\
&- \Delta H =- k_2 \frac{HF}{K_2 + F} - \rho_2 H\hspace{1em} &\text{in } \Omega,\label{e2}\\
&-D\Delta F - \nabla F\cdot \nabla p = k_1\frac{(M_0-F)L}{K_1+L}-k_2\frac{HF}{K_2+F}-\lambda\frac{F(M_0-F)L}{M_0(\gamma+H)}+(\rho_3-\rho_4)\frac{(M_0-F)F}{M_0}\hspace{0.5em} &\text{in }\Omega,\label{e3}\\
&-\Delta p = \frac{1}{M_0}\Big[\lambda\frac{(M_0-F)L}{\gamma+H}-\rho_3(M_0-F) - \rho_4 F\Big]\hspace{1em} &\text{in }\Omega,\label{e4}\\
&\frac{\partial L}{\partial r} = \frac{\partial H}{\partial r} = \frac{\partial F}{\partial r} =  \frac{\partial p}{\partial r} = 0 \hspace{2em}&  r=1,\label{e5}\\
&\frac{\partial L}{\partial {\bf n}} + \beta_1 (L-L_0) = 0, \hspace{1em}  \frac{\partial H}{\partial {\bf n}} + \beta_1 (H-H_0)=0, \hspace{1em} \frac{\partial F}{\partial {\bf n}} + \beta_2 F = 0&\text{on }\Gamma,\label{e6}\\
&p = \kappa &\text{on }\Gamma,\label{e7}\\
&V_n = -\frac{\partial p}{\partial {\bf n}}=0 \hspace{2em}&\text{on }\Gamma.\label{e8}
\end{eqnarray}

In recent years, considerable research works have been carried out on bifurcation analysis for various tumor growth models (see \cite{CE3, FFBessel, FH3, FR2, Hao1, Hao2, angio1, Fengjie, Hongjing, Zejia, W, WZ2,zhao2,song}), where the Crandall-Rabinowitz theorem (will be mentioned in Section 2) is a primary tool. Compared with tumor growth models, our system \re{e1} -- \re{e8} contains more equations which are highly nonlinear and coupled together, therefore it is a formidable task to analyze our model. Besides, the absence of an explicit stationary solution presents a big challenge to verify the Crandall-Rabinowitz theorem. Even though the problems in \cite{zhao, zhao2} do not admit explicit representations, the structure of the problem studied here is very different. To overcome it, we establish a lot of sharp estimates in Section 4. To the best of our knowledge, this is the first paper on the study of bifurcation for the system \re{p1} -- \re{p13}. Our main result is stated as follows:

For convenience we shall use $\mu= \frac1\epsilon[\lambda L_0-\rho_3(\gamma+H_0)]$ as our bifurcation parameter. We will
keep all parameters fixed except $L_0$ and $\rho_4$, and vary $\mu$
by changing $L_0$.

\begin{thm}\label{result} 
For each integer $n\ge 2$, we can find a small $E>0$ and for each $0<\epsilon<E$, there exists a unique  $\mu_n=(\gamma+H_0)n^2(1-n^2)+ O(n^5 \epsilon)$
such that if $\mu_n > \mu_c$ ($\mu_c$ is defined in (\ref{muc})), then $\mu=\mu_n$ is a bifurcation point of the symmetry-breaking stationary solution of the system \re{e1} -- \re{e8}. Moreover, the free boundary of this bifurcation solution is of the form
\begin{equation*}
    r = 1 - \epsilon + \tau \cos(n\theta) + o(\tau), \hspace{1em} \text{where }\hspace{1em} |\tau|\ll \epsilon.
\end{equation*}
\end{thm}

\begin{rem}
Unlike tumor protrusions which are usually unstable and may cause metastases, the protrusions of plaques are towards the blood region with limited {\bf spatial freedom}. As $n$ gets bigger, $\mu_n$ becomes negative
with larger absolute value. By the definition of $\mu_n$, this means
that the concentration of the {\bf good cholesterol (HDL)} must be substantially larger than the concentration of the {\bf bad cholesterol (LDL)} for the bifurcation to occur. The more protrusions, the larger $H_0$ over $L_0$ will be required to balance the protrusion forces. Based on the stability results from \cite{FHHplaque}, it is likely to have some stable bifurcation branches.
\end{rem}

 The structure of this paper is as follows. In Section 2, we give some preliminaries; in section 3, we rigorously justify some expansions which will be needed in applying the Crandall-Rabinowitz theorem; and then we carry out our proof of Theorem \ref{result} in Section 4. Some well-known results 
are collected in the Appendix.

\section{Radially symmetric stationary solution}

\subsection{A small radially symmetric stationary solution}
We consider a radially symmetric stationary solution in a small ring-region $\Omega_*=\{1-\epsilon < r < 1\}$, and denote the solution by $(L_*, H_*, F_*, p_*)$. Based on \re{e1} -- \re{e8}, the solution satisfies
\begin{eqnarray}
&- \Delta L_* = - k_1 \frac{(M_0-F_*)L_*}{K_1 + L_*} - \rho_1 L_* &\text{in }\Omega_*,\label{r1}\\
&- \Delta H_* =- k_2 \frac{H_*F_*}{K_2 + F_*} - \rho_2 H_* &\text{in }\Omega_*,\label{r2}\\
&\label{r3}
    -D\Delta F_* - \frac{\p F_*}{\p r}\frac{\p p_*}{\p r}= k_1\frac{(M_0-F_*)L_*}{K_1+L_*}-k_2\frac{H_*F_*}{K_2+F_*}-\lambda\frac{F_*(M_0-F_*)L_*}{M_0(\gamma+H_*)}+(\rho_3-\rho_4)\frac{(M_0-F_*)F_*}{M_0} &\text{in }\Omega_*,\\
&-\Delta p_* = \frac{1}{M_0}\Big[\lambda\frac{(M_0-F_*)L_*}{\gamma+H_*}-\rho_3(M_0-F_*) - \rho_4 F_*\Big]\hspace{1em} &\text{in }\Omega_*,\label{r4}\\
&\frac{\partial L_*}{\partial r} = \frac{\partial H_*}{\partial r} = \frac{\partial F_*}{\partial r} =  \frac{\partial p_*}{\partial r} = 0, \hspace{2em}& r=1,\label{r5}\\
&\m -\frac{\partial L_*}{\partial r} + \beta_1 (L_*-L_0) = 0, \hspace{1em}  -\frac{\partial H_*}{\partial r} + \beta_1 (H_*-H_0)=0, \hspace{1em} -\frac{\partial F_*}{\partial r} + \beta_2 F_* = 0,\hspace{1em}&r=1-\epsilon,\label{r6}\\
&p_* = -\frac{1}{1-\epsilon},\hspace{1em}  &r=1-\epsilon,\label{r7}\\
&  \frac{\partial p_*}{\partial r}=0, &r=1-\epsilon.\label{r7a}
\end{eqnarray}
Viewing $\frac{\p p_*}{\p r}$ as $-v$, and following \textit{Theorem 3.1} in \cite{FHHplaque}, for every $H_0=O(1)$ and $\epsilon$ small, we can find 
a unique $L_0$ and a constant $K_*$, such that there is a unique classical solution to the above system with $|\lambda L_0 - 
\rho_3 (H_0+\gamma)| < K_* \epsilon$. The existence theorem for radially symmetric solution of this form, however, {\em is not good enough} for the
bifurcation theorem.

There are many parameters in our system. We need to choose one as the bifurcation parameter. We let $\mu= \frac1\epsilon [\lambda L_0 - \rho_3 (\gamma+ H_0)]$ to be our bifurcation parameter. We can vary $\mu$ by,
say, keeping $\lambda, \gamma, \rho_3, H_0$ and $\epsilon$ fixed
while changing $L_0$ only. For simplicity, we shall assume all the parameters are fixed and of order $O(1)$ except $L_0$ and $\rho_4$. With these settings, 
varying $L_0$ corresponds to varying $\mu$. In the rest of this paper, \textit{we shall thus use
$\mu$ and $\rho_4$ as our parameters}.

Here is our existence theorem for the radially symmetric solutions.
We define
\be\label{muc}
 \mu_c = \frac{ \rho_3}{\beta_1}\Big\{  
 (\gamma + H_0) \Big( \frac{\lambda k_1 M_0  }{\lambda K_1+\rho_3(\gamma+H_0)}+ {\rho_1 } \Big) -\rho_2 H_0 \Big\}.
\ee
\begin{thm}\label{thm21}
 For every $\mu^*>\mu_c$ and $\mu_c<\mu<\mu^*$, we can find a small $\epsilon^*>0$, 
 and for each $0<\epsilon<\epsilon^*$, there exists a unique $\rho_4$ such
 that the system \re{r1} -- \re{r7a} admits a unique solution $(L_*, H_*, F_*, p_*)$. 
\end{thm}
\begin{proof}
The proof is similar to that in \cite{FHHplaque} but much more involved.
Following \textit{Lemma 3.1} of \cite{FHHplaque}, for all parameters of order $O(1)$,
the system \re{r1} -- \re{r7} admits a unique solution for small $\epsilon$.
In order for this solution to be the solution of our problem, we need to
verify \re{r7a}. We shall do so by keeping all parameters fixed except 
$\rho_4$.

Note that \re{r7a} is equivalent to 
\be \label{fb}
\Phi(\rho_4,\epsilon, \mu) = 0, \mm  \text{where }
\Phi(\rho_4,\epsilon, \mu) \triangleq \int_{1-\epsilon}^1 \Big[\lambda\frac{(M_0-F_*)L_*}{\gamma+H_*}-\rho_3(M_0-F_*) - \rho_4 F_*\Big] r \dif r.
\ee
As in \cite[(3.29)--(3.32)]{FHHplaque}, recalling also 
(see Appendix \ref{app-super})
$\xi(r) =  \frac{1-r^2}4 + \frac12\log r =O(\epsilon^2)$
(the formulas in  \cite[(3.23)--(3.25), (3.26)--(3.28), (3.29)]{FHHplaque}
are all missing minus signs; as a result, the corrected 
\cite[(3.29)]{FHHplaque} should read:
\renewcommand{\theequation}{\bf [9,(3.29)]}
\be
L_*(r)  =  L_0 -  
 \Big( \frac{k_1 M_0 L_0}{K_1+L_0}+\rho_1 L_0\Big)\Big(
 \xi(r)+\frac\epsilon{\beta_1}\Big)+ \text{Const}\cdot \epsilon^2
 +O(\epsilon^3),
\ee
\addtocounter{equation}{-1}
\renewcommand{\theequation}{\thesection.\arabic{equation}}
\hspace{-1.24em} and \cite[(3.30),(3.31)]{FHHplaque} should be corrected in a similar manner; this correction does not change the proof in \cite{FHHplaque}),
 we can establish the following:
\begin{eqnarray}
 \label{Ls} L_*(r) & = & L_0 - \frac{\epsilon}{\beta_1}
 \Big( \frac{k_1 M_0 L_0}{K_1+L_0}+\rho_1 L_0\Big) + O(\epsilon^2) \\
 & = & \frac{\rho_3(\gamma+H_0)}\lambda + \epsilon \Big[ \frac\mu\lambda  \nonumber
  - \frac{\rho_3(\gamma+H_0)}{\beta_1}
 \Big( \frac{k_1 M_0  }{\lambda K_1+\rho_3(\gamma+H_0)}+\frac{\rho_1 }\lambda  \Big)\Big]+  O(\epsilon^2) \\
 & \triangleq & \frac{\rho_3(\gamma+H_0)}\lambda +\epsilon L_*^1 + O(\epsilon^2),\nonumber  \\
 \label{Hs} H_*(r) & = & H_0  -   \epsilon  \frac{\rho_2 H_0 }{\beta_1} +  O(\epsilon^2)
\; \triangleq \; H_0 + \epsilon H_*^1 + O(\epsilon^2),\\
 \label{Fs} F_*(r) & = &  \frac\epsilon{\beta_2}  \frac{k_1 M_0 L_0}{D(K_1+L_0)}   +  O(\epsilon^2) \\
 & = & \epsilon \; \frac{\rho_3(\gamma+H_0)}{\beta_2 D} \;
 \frac{k_1 M_0  }{\lambda K_1+\rho_3(\gamma+H_0)}+ O(\epsilon^2) 
\; \triangleq \; \epsilon F_*^1 + O(\epsilon^2).\nonumber
\end{eqnarray}
Substituting these expressions into the formula \re{fb} for $\Phi$,
we find that the $O(1)$ terms in the bracket $[\cdots]$  cancel out, and  
\be
 \Phi(\rho_4,\epsilon, \mu) = \int_{1-\epsilon}^1 \bigg\{ \epsilon \Big[
 \frac{M_0(\lambda L_*^1-\rho_3 H_*^1)}{\gamma+H_0} -\rho_4 F_*^1\Big] + O(\epsilon^2) \bigg\} rdr.
\ee
A direct computation shows that 
\bea\label{mum}
 \frac{M_0(\lambda L_*^1-\rho_3 H_*^1)}{\gamma+H_0}  
 & = & \frac{M_0}{\gamma + H_0} (\mu - \mu_c).
\eea
It follows that, for small $\epsilon$, 
$ \Phi(0,\epsilon, \mu) > 0 $ and  $ \Phi(\rho_4,\epsilon, \mu)<0$
for large $\rho_4$, hence there must be a value of $\rho_4$ on which 
$ \Phi(\rho_4,\epsilon, \mu) = 0$.

To finish the proof, it suffices to show $\frac\p{\p \rho_4} \Phi(\rho_4,\epsilon) < 0$;
the proof is similar to that 
of \cite[Theorem 3.1]{FHHplaque} in the second part, but is actually
a little easier.
\end{proof}

\begin{rem}
By ODE theories, the solution $(L_*, H_*, F_*, p_*)$ can be extended to the bigger region $\Omega_{2\epsilon}=\{1-2\epsilon <r <1\}$ while maintaining $C^\infty$ regularity. For notational convenience, we still use $(L_*, H_*, F_*, p_*)$ to denote the extended solution. 
\end{rem}

\begin{rem} The case $\mu_c<0$ is certainly true within reasonable parameter range. 
\end{rem}

Following the above proof, we also derive
\begin{lem}
Let $\mu > \mu_c$. Then
\bea
  \rho_4 & = & \frac1{F_*^1}\;  \frac{M_0}{\gamma + H_0} (\mu-\mu_c) + O(\epsilon) \;=\;\frac{\beta_2 D [ \lambda K_1 + \rho_3(\gamma+ H_0)]}{\rho_3 k_1(\gamma+H_0)^2} (\mu-\mu_c) + O(\epsilon),\label{rho4}\\
  \frac{\p \rho_4}{\p \mu} & = & \frac1{F_*^1}\;  \frac{M_0}{\gamma + H_0} + O(\epsilon)  \;=\;
  \frac{\beta_2 D [ \lambda K_1 + \rho_3(\gamma+ H_0)]}{\rho_3 k_1(\gamma+H_0)^2}+ O(\epsilon).\label{rho4D}
\eea
\end{lem}

\begin{rem}
In contrast to \cite{angio1, Zejia, zhao2}, where $\tilde \sigma$ is independent of $\mu$,
here the explicit dependence of $\rho_4$ with respect to $\mu$ is given 
in the above lemma.
\end{rem}

The following estimates are useful later on:
\begin{lem} The following estimate holds for first derivatives,
\be
 \label{deri} |L_*'(r)| + |H_*'(r)| + |F_*'(r)| + |p_*'(r)|\le C \epsilon,
 \m 1-\epsilon\le r\le 1.
\ee
\end{lem}
\begin{proof}
From \re{Ls}--\re{Fs} we derive that $|\Delta L_*|\le C$,
$|\Delta H_*|\le C$, $|\Delta p_*|\le C$. Using the boundary condition
$L_*'(1)=0$, we find that
\[
 | rL_*'(r) | = \Big|\int_r^1 (\xi L_*'(\xi))' d\xi \Big| \le C\epsilon,
  \m 1-\epsilon\le r \le 1.
 \]
 The estimates for $H_*'(r)$  and for $p_*'(r)$ are similar.
 Finally, for $F_*'(r)$, using the above estimates we find
 \[
  |(rF_*'(r))'|\le C + \frac{C\epsilon}{D} \max_{1-\epsilon\le r\le 1}
  |rF_*'(r)|.
 \]
We then integrate over $(r,1)$ and use $F_*'(1)=0$ to derive
 \[
  | rF_*'(r) | \le C\epsilon + \frac{C\epsilon^2}{D} \max_{1-\epsilon\le r\le 1}
  |rF_*'(r)|,
 \]
which implies $| rF_*'(r) | \le C\epsilon$.
\end{proof}

\subsection{The Crandall-Rabinowitz theorem} Next we state a useful theorem which is  critical in studying bifurcations.

\begin{thm}\label{CRthm}  {\bf (Crandall-Rabinowitz theorem, \cite{crandall})}
Let $X$, $Y$ be real Banach spaces and $\mathcal{F}(\cdot,\cdot)$ a $C^p$ map, $p\ge 3$, of a neighborhood $(0,\mu_0)$ in $X \times \mathbb{R}$ into $Y$. Suppose
\begin{itemize}
\item[(1)] $\mathcal{F}(0,\mu) = 0$ for all $\mu$ in a neighborhood of $\mu_0$,
\item[(2)] $\mathrm{Ker} \,\mathcal{F}_x(0,\mu_0)$ is one dimensional space, spanned by $x_0$,
\item[(3)] $\mathrm{Im} \,\mathcal{F}_x(0,\mu_0)=Y_1$ has codimension 1,
\item[(4)] $\mathcal{F}_{\mu x}(0,\mu_0) x_0 \notin Y_1$.
\end{itemize}
Then $(0,\mu_0)$ is a bifurcation point of the equation $\mathcal{F}(x,\mu)=0$ in the following sense: In a neighborhood of $(0,\mu_0)$ the set of solutions $\mathcal{F}(x,\mu) =0$ consists of two $C^{p-2}$ smooth curves $\Gamma_1$ and $\Gamma_2$ which intersect only at the point $(0,\mu_0)$; $\Gamma_1$ is the curve $(0,\mu)$ and $\Gamma_2$ can be parameterized as follows:
$$\Gamma_2: (x(\epsilon),\mu(\epsilon)), |\epsilon| \text{ small, } (x(0),\mu(0))=(0,\mu_0),\; x'(0)=x_0.$$
\end{thm}

\section{Bifurcations - preparations}
Let's consider a family of perturbed domains $\Omega_\tau=\{1-\epsilon+ \tilde{R} < r < 1\}$ and denote the corresponding inner boundary to be $\Gamma_\tau$, where $\tilde{R}=\tau S(\theta)$, $|\tau| \ll \epsilon$ and $|S|\le 1$. Let $(L,H,F,p)$ be the solution of
\begin{eqnarray} 
&- \Delta L = - k_1 \frac{(M_0-F)L}{K_1 + L} - \rho_1 L\hspace{1em} &\text{in } \Omega_\tau,\label{b1}\\
&- \Delta H =- k_2 \frac{HF}{K_2 + F} - \rho_2 H\hspace{1em} &\text{in } \Omega_\tau,\label{b2}\\
&-D\Delta F - \nabla F\cdot \nabla p = k_1\frac{(M_0-F)L}{K_1+L}-k_2\frac{HF}{K_2+F}-\lambda\frac{F(M_0-F)L}{M_0(\gamma+H)}+(\rho_3-\rho_4)\frac{(M_0-F)F}{M_0}\hspace{1em} &\text{in }\Omega_\tau,\label{b3}\\
&-\Delta p = \frac{1}{M_0}\Big[\lambda\frac{(M_0-F)L}{\gamma+H}-\rho_3(M_0-F) - \rho_4 F\Big]\hspace{1em} &\text{in }\Omega_\tau,\label{b4}\\
&\frac{\partial L}{\partial r} = \frac{\partial H}{\partial r} = \frac{\partial F}{\partial r} =  \frac{\partial p}{\partial r} = 0, \hspace{2em}& r=1,\label{b5}\\
&\frac{\partial L}{\partial {\bf n}} + \beta_1 (L-L_0) = 0, \hspace{1em}  \frac{\partial H}{\partial {\bf n}} + {\beta_1}(H-H_0)=0, \hspace{1em} \frac{\partial F}{\partial {\bf n}} + \beta_2 F = 0&\text{on } \Gamma_\tau,\label{b6}\\
&p = \kappa &\text{on } \Gamma_\tau.\label{b7}
\end{eqnarray}
The existence and uniqueness of such a solution is guaranteed by the following lemma.

\begin{lem}\label{lem3.1}
Let $S\in C^{4+\alpha}(\Sigma)$ ($\Sigma$ denotes the unit closed disk) with $\|S\|_{C^{4+\alpha}(\Sigma)}\le 1$. For sufficiently small $\epsilon$ and $|\tau|\ll\epsilon$, there is a unique solution $(L,H,F,p)$ to the problem \re{b1} -- \re{b7}.
\end{lem}
\begin{proof}
We shall use the contraction mapping principle to prove this lemma. Let
\begin{equation}\label{contraction}
    \mathcal{M}=\{(L,H,F);\; 0\le L\le L_0,\, 0\le H\le H_0,\, 0\le F\le M_0\}.
\end{equation}

\noindent \textbf{\textit{Step 1.}}
For each $(L,M,F)\in \mathcal{M}$, we define a map $\mathcal{L}:(L,H,F)\rightarrow (\hat{L},\hat{H},\hat{F})$ as follows: we first solve $\hat{L}$ and $\hat{H}$ from the elliptic equations
\begin{eqnarray*}
    &- \Delta \hat{L} = - k_1 \frac{(M_0-F)\hat{L}}{K_1 + L} - \rho_1 \hat{L}\hspace{2em} &\text{in } \Omega_\tau,\\
&- \Delta \hat{H} =- k_2 \frac{\hat{H}F}{K_2 + F} - \rho_2 \hat{H}\hspace{2em} &\text{in } \Omega_\tau,
\end{eqnarray*}
with the
boundary conditions
\begin{eqnarray*}
    &\frac{\p \hat{L}}{\p r}=\frac{\p \hat{H}}{\p r}=0, \hspace{2em} & r=1,\\
    &\frac{\partial \hat{L}}{\partial {\bf n}} + {\beta_1}(\hat{L}-L_0) = \frac{\partial \hat{H}}{\partial {\bf n}} + {\beta_1}(\hat{H}-H_0)=0\hspace{2em}&\text{on }\Gamma_\tau.
\end{eqnarray*}
By the maximum principle, we clearly have
\begin{equation}
    0\le \hat{L} \le L_0,\hspace{2em} 0\le \hat{H}\le H_0 \hspace{2em} \text{in }\bar{\Omega}_\tau. \label{lem31-1}
\end{equation}

We then define $\hat{p}$ by the solution of the system
\begin{gather}
\label{lemp}
    -\Delta \hat{p} = \frac{1}{M_0}\Big[\lambda\frac{(M_0-F)L}{\gamma+H}-\rho_3(M_0-F) - \rho_4 F\Big]\hspace{2em} \text{in }\Omega_\tau,\\
    \label{lempbc}
    \frac{\p \hat{p}}{\p r}\Big|_{r=1} = 0,\hspace{2em} \hat{p}\Big|_{\Gamma_\tau} = \kappa .
\end{gather}
Since $L,H,F$ are all bounded, the right-hand side of \re{lemp} is bounded under supremum norm, i.e.,
\begin{equation}
\label{boundp}
    |\Delta (\hat{p}+1)| \le C.
\end{equation}
Also, we use the mean-curvature formula, i.e.,
\begin{equation}
    \label{kappa}
    \kappa|_{\Gamma_\tau} = -\frac{1}{1-\epsilon}+\frac{\tau}{(1-\epsilon)^2}(S+S_{\theta\theta})+\tau^2 f_1, \hspace{1em} \text{where } \|f_1\|_{C^{1+\alpha}}\le C\|S\|_{C^{3+\alpha}(\Sigma)},
\end{equation}
to derive that 
\begin{equation}
\label{boundpbc}
    \|\hat{p}+1\|_{C^{1+\alpha}(\Gamma_\tau)} \le C\epsilon.
\end{equation}
By the maximum principle,
\begin{equation}
    \label{hatpbound}
    \|\hat{p}+1\|_{L^\infty(\Omega_\tau)} \le C(\xi+\epsilon)\le C(O(\epsilon^2)+\epsilon)\le C\epsilon,
\end{equation}
where $\xi$ is defined in Appendix \ref{app-super}. Next we are going to estimate $\|\hat{p}\|_{C^1}$ and show that it is actually independent of $\epsilon$ and $\tau$. To do that, we shall use the Schauder estimates; but before using the Schauder estimates directly, let's apply the following transformation:
\begin{equation*}
    T_\tau: \tilde{r}=\frac{r-1}{2(\epsilon-\tau S(\theta))}+1, \hspace{2em} \tilde{\theta}=\frac{\theta}{\epsilon},
\end{equation*}
and denote $\tilde{p}(\tilde{r},\tilde{\theta})=\hat{p}(r,\theta)-1$. Clearly, $T_\tau$ maps $\Omega_\tau$ to a long stripe region $(\tilde{r},\tilde{\theta})\in[\frac12,1]\times[0,\frac{2\pi}{\epsilon}]$. Based on the calculations from Appendix \ref{app-2}, $\tilde{p}$ satisfies
\begin{equation*}
     -\frac{\p }{\p \tilde{r}}\Big((1+ A_1) \frac{\p \tilde{p}}{\p \tilde{r}} + A_2 \frac{\p \tilde{p}}{\p \tilde{\theta}}\Big) - \frac{\p}{\p \tilde{\theta}} \Big(A_3\frac{\p \tilde{p}}{\p \tilde{r}} + (1+A_4) \frac{\p \tilde{p}}{\p \tilde{\theta}}\Big) +  A_5\frac{\p \tilde{p}}{\p \tilde{r}} + A_6\frac{\p \tilde{p}}{\p \tilde{\theta}} = \epsilon^2 f_2,
\end{equation*}
where coefficients $A_1,A_2,A_3,A_4\in C^{\alpha}$, $A_5,A_6$ are bounded, and $f_2=\frac{r}{M_0}\Big[\lambda\frac{(M_0-F)L}{\gamma+H}-\rho_3(M_0-F) - \rho_4 F\Big]$ is also bounded based on \re{contraction}. Applying the interior sub-Schauder estimates (Theorem 8.32, \cite{GT}) on the region $\Omega_{i_0}: (\tilde{r},\tilde{\theta})\in [\frac12,1]\times [\theta_{i_0}-2,\theta_{i_0}+2]$, recalling also \re{boundpbc}, we obtain
\begin{equation*}
\begin{split}
    \|\tilde{p}\|_{C^{1+\alpha}([\frac12,1]\times [\theta_{i_0}-1,\theta_{i_0}+1])}&\le C_{i_0}\Big(\epsilon^2 \|f_2\|_{L^\infty(\Omega_{i_0})} + \|\tilde{p}\|_{L^\infty(\Omega_{i_0})} + \|\t p\|_{C^{1+\alpha}(\{\t r=\frac12\})}\Big)\\
    &\le C_{i_0}\Big(\epsilon^2\|f_2\|_{L^\infty([\frac12,1]\times[0,\frac{2\pi}{\epsilon}])} + \|\hat{p}+1\|_{L^\infty(\Omega_\tau)} + \|\hat{p}+1\|_{C^{1+\alpha}(\Gamma_\tau)}\Big)\\
    &\le \tilde{C}_{i_0}\epsilon,
    \end{split}
\end{equation*}
where $\tilde{C}_{i_0}$ is independent of $\epsilon$ and $\tau$. We use a series of sets $[\frac12,1]\times[\theta_{i_0}-1,\theta_{i_0}+1]$ to cover the whole region $[\frac12,1]\times[0,\frac{2\pi}{\epsilon}]$, as a result, 
$$\|\tilde{p}\|_{C^{1+\alpha}([\frac12,1]\times[0,\frac{2\pi}{\epsilon}])}\le C\epsilon.$$
We then relate $\tilde{p}$ with $\hat{p}$ to derive
\begin{equation*}
    \|\hat{p}+1\|_{C^1(\bar{\Omega}_\tau)}\le \frac{1}{\epsilon}\|\tilde{p}\|_{C^1([\frac12,1]\times[0,\frac{2\pi}{\epsilon}])}\le \frac{1}{\epsilon}\|\tilde{p}\|_{C^{1+\alpha}([\frac12,1]\times[0,\frac{2\pi}{\epsilon}])}\le C,
\end{equation*}
and hence 
\begin{equation}
    \label{p-reg}
    \|\nabla \hat{p}\|_{L^\infty(\Omega_\tau)} \le C,
\end{equation}
where $C$ is independent of $\epsilon$ and $\tau$.

Finally, recalling equation \re{b3}, we define $\hat{F}$ as the solution to the equation
\begin{equation}
    \label{lem31-2}
        -D\Delta \hat{F} - \nabla \hat{F}\cdot \nabla \hat{p} = k_1\frac{(M_0-\hat{F})L}{K_1+L}-k_2\frac{H\hat{F}}{K_2+F}-\lambda\frac{\hat{F}(M_0-F)L}{M_0(\gamma+H)}
        +\frac{\rho_3}{M_0}(M_0-\hat{F})F - \frac{\rho_4}{M_0}(M_0-F)\hat{F},
\end{equation}
with the boundary conditions
\begin{equation}
\label{lem31-3}
    \frac{\p \hat{F}}{\p r}\Big|_{r=1} = 0, \hspace{2em} \Big[\frac{\p \hat{F}}{\p {\bm n}}+\beta_2 \hat{F}\Big]\Big|_{\Gamma_\tau} = 0.
\end{equation}
By the maximum principle,
$
    \hat{F}\ge 0 \hspace{.5em} \text{in }\bar{\Omega}_\tau,
$
and, using this result, we employ the maximum principle again to derive the inequality
$
    M_0 - \hat{F} \ge 0 \hspace{.5em} \text{in } \bar{\Omega}_\tau.
$ 
All together, these two inequalities indicate
\begin{equation}\label{boundforp}
    0\le \hat{F}\le M_0.
\end{equation}
In the next step, we claim that this bound for $\hat{F}$ can be improved. By \re{contraction} and \re{boundforp}, the right-hand side of equation \re{lem31-2} is bounded; assume the bound is constant $C$. According to Appendix~\ref{app-super}, $C(\xi(r)+c_1(\beta_2,\epsilon)+c_2(\beta_2,\tau))$ can be a supersolution for $\hat{F}$, hence the maximum principle leads to
\begin{equation}\label{hatFbound}
    \Big\|\hat{F}\Big\|_{L^\infty(\Omega_\tau)} \le \Big\|C\Big(\xi(r)+c_1(\beta_2,\epsilon)+c_2(\beta_2,\tau)\Big)\Big\|_{L^\infty(\Omega_\tau)}\le C\Big(\frac{\epsilon}{\beta_2}+ \frac{2}{\beta_2}|\tau| + O(\epsilon^2)\Big) \le C\epsilon.
\end{equation}
After we show this, we can employ the sub-Schauder estimates on \re{lem31-2} -- \re{lem31-3} in a similar way as we did for $\hat{p}$ to obtain
\begin{equation}
    \label{F-reg}
    \|\nabla \hat{F}\|_{L^\infty(\Omega_\tau)} \le C,
\end{equation}
where $C$ is a constant which does not depend upon $\epsilon$ and $\tau$.

Above, we have shown that $(\hat{L}, \hat{H}, \hat{F})\in \mathcal{M}$, which means $\mathcal{L}$ maps $\mathcal{M}$ into itself. We shall next prove that $\mathcal{L}$ is a contraction.

\vspace{10pt}
\noindent \textbf{\textit{Step 2.}}
Suppose that $(\hat{L}_j, \hat{H}_j, \hat{F}_j)=\mathcal{L}(L_j, H_j, F_j)$ for $j=1,2$, and set
\begin{gather*}
    \mathcal{A} = \|L_1-L_2\|_{L^\infty(\Omega_\tau)} + \|H_1-H_2\|_{L^\infty(\Omega_\tau)} + \|F_1-F_2\|_{L^\infty(\Omega_\tau)},\\
    \mathcal{B} = \|\hat{L}_1-\hat{L}_2\|_{L^\infty(\Omega_\tau)} + \|\hat{H}_1-\hat{H}_2\|_{L^\infty(\Omega_\tau)} + \|\hat{F}_1-\hat{F}_2\|_{L^\infty(\Omega_\tau)}.
\end{gather*}
Based on our definitions of $\hat{L}_j$, $\hat{H}_j$, $\hat{p}_j$, $\hat{F}_j$ in the first step and recalling \re{p-reg} as well as \re{F-reg}, we derive, for some constant $C^*$,
\begin{equation*}
\begin{split}
    &|\Delta(\hat{L}_1-\hat{L}_2)|\le C^*(\mathcal{A}+\mathcal{B}), \mm
    |\Delta(\hat{H}_1-\hat{H}_2)|\le C^*(\mathcal{A}+\mathcal{B}),\\
    &|\nabla \hat{F}_1| + |\nabla \hat{F}_2| \le C^*, \quad
    |\nabla \hat{p}_1| + |\nabla \hat{p}_2| \le C^*, \quad|\nabla(\hat{p}_1 - \hat{p}_2)|\le C^*\mathcal{A},\\
    &|D\Delta(\hat{F}_1-\hat{F}_2) + \nabla\hat{p}_1 \cdot \nabla(\hat{F}_1 - \hat{F}_2)| \le C^*(\mathcal{A}+\mathcal{B}).
\end{split}
\end{equation*}
Next we shall use the maximum principle to derive bounds for $\hat{L}_1-\hat{L}_2$, $\hat{H}_1-\hat{H}_2$, and $\hat{F}_1-\hat{F}_2$. To do that, we use the function $\xi(r) + c_1(\beta_1,\epsilon) + c_2(\beta_1,\tau)$ defined in Appendix \ref{app-super}. As a result,
\begin{gather*}
    |\hat{L}_1 - \hat{L}_2| \le C^*(\mathcal{A}+\mathcal{B})(\xi+c_1(\beta_1,\epsilon) + c_2(\beta_1,\tau))\Rightarrow \|\hat{L}_1 - \hat{L}_2\|_{L^\infty(\Omega_\tau)}\le C^{**}(\mathcal{A}+\mathcal{B})(\epsilon+|\tau|),\\
    |\hat{H}_1 - \hat{H}_2| \le C^*(\mathcal{A}+\mathcal{B})(\xi+c_1(\beta_1,\epsilon) + c_2(\beta_1,\tau))\Rightarrow \|\hat{H}_1 - \hat{H}_2\|_{L^\infty(\Omega_\tau)}\le C^{**}(\mathcal{A}+\mathcal{B})(\epsilon+|\tau|),\\
    |\hat{F}_1 - \hat{F}_2| \le C^*(\mathcal{A}+\mathcal{B})(\xi+c_1(\beta_2,\epsilon) + c_2(\beta_2,\tau))\Rightarrow \|\hat{F}_1 - \hat{F}_2\|_{L^\infty(\Omega_\tau)}\le C^{**}(\mathcal{A}+\mathcal{B})(\epsilon+|\tau|),
\end{gather*}
where both $C^*$ and $C^{**}$ are independent of $\epsilon$ and $\tau$. The above inequalities imply that
$$\mathcal{B}\le C^{**}(\mathcal{A}+\mathcal{B})(\epsilon+|\tau|),$$
hence we obtain a contraction mapping by taking $\epsilon$ sufficiently small and $|\tau|\ll \epsilon$ so that
$$\frac{C^{**}(\epsilon+|\tau|)}{1-C^{**}(\epsilon+|\tau|)}< 1.$$
The unique fixed point of the contraction mapping is the unique
classical solution to the system \re{b1} -- \re{b7}.
\end{proof}

With $p$ being uniquely determined in the system \re{b1} -- \re{b7}, we define $\mathcal{F}$ by
\begin{equation}
    \label{F}
    \mathcal{F}(\tau S,\mu) = -\frac{\p p}{\p {\bf n}}\Big|_{\Gamma_\tau},
\end{equation}
where $\mu = \frac1\epsilon [\lambda L_0 - \rho_3 (\gamma+ H_0)]$ is our bifurcation parameter. We know that $(L,H,F,p)$ is a symmetry-breaking stationary solution if and only if $\mathcal{F}(\tau S,\mu)=0$.

In order to apply the Crandall-Rabinowitz theorem, we need to compute the Fr\'echet derivatives of $\mathcal{F}$. For a fixed small $\epsilon$, we formally write $(L,H,F,p)$ as
\begin{gather}
    L = L_*+\tau L_1 + O(\tau^2),\label{expand1}\\
    H = H_*+\tau H_1 + O(\tau^2), \label{expand2}\\
    F = F_*+\tau F_1 + O(\tau^2), \label{expand3}\\
    p = p_*+\tau p_1 + O(\tau^2). \label{expand4}
\end{gather}
In the following, we shall first justify \re{expand1} -- \re{expand4}. The structure of the proofs is  similar to that in \cite{CE3, angio1, Fengjie, Hongjing, Zejia, zhao2}. However, our problem is much more involved since the system \re{b1} -- \re{b7} is highly nonlinear and coupled, hence we shall use very delicate estimates and the continuation lemma (see Appendix \ref{appcon}) to tackle the problem.

\subsection{First-order $\tau$ estimates}
\begin{lem}\label{first}
  Fix $\epsilon$ sufficiently small, if $|\tau|\ll \epsilon$ and $\|S\|_{C^{4+\alpha}(\Sigma)}\le 1$, then we have
 \begin{eqnarray*}
     &\max\{\|L-L_*\|_{L^\infty(\Omega_\tau)} , \|H-H_*\|_{L^\infty(\Omega_\tau)}, \|F-F_*\|_{L^\infty(\Omega_\tau)}, \|p-p_*\|_{L^\infty(\Omega_\tau)}\} \le C|\tau|\|S\|_{C^{4+\alpha}(\Sigma)},\\
     & \max\{ \|\nabla(F-F_*)\|_{L^\infty(\Omega_\tau)}, \|\nabla(p-p_*)\|_{L^\infty(\Omega_\tau)}\}\le \frac{C}\epsilon|\tau|\|S\|_{C^{4+\alpha}(\Sigma)},
 \end{eqnarray*}
 where $C$ is independent of $\epsilon$ and $\tau$.
\end{lem}
\begin{proof}
We combine \re{r1} -- \re{r7} and \re{b1} -- \re{b7} to obtain the equations for $L-L_*$, $H-H_*$, $F-F_*$ and $p-p_*$. For example, we have
\begin{eqnarray}
\nonumber
        -\Delta(L-L_*) &=& -k_1\frac{(M_0-F)L}{K_1+L}-\rho_1 L + k_1\frac{(M_0-F_*)L_*}{K_1+L_*}+\rho_1 L_*\\
        &=& \Big[-k_1\frac{(M_0-F)K_1}{(K_1+L)(K_1+L_*)}-\rho_1\Big](L-L_*) + k_1 \frac{L_*}{K_1+L_*}(F-F_*)\label{lls}\\
        &\triangleq& b_1(r)(L-L_*) + b_2(r)(F-F_*),\nonumber
\end{eqnarray}
where $b_1(r)$ and $b_2(r)$ are both bounded since $0\le L_*,L\le L_0$, $0\le H_*,H \le H_0$, and $0\le F,F_*\le M_0$ based on Lemma \ref{lem3.1} and \textit{Lemma 3.1} in \cite{FHHplaque}. In addition, the boundary conditions for $L-L_*$ are
\begin{gather*}
    \frac{\p (L-L_*)}{\p r}\Big|_{r=1} = 0,\label{newbd1}\\
    \Big(\frac{\p (L-L_*)}{\p {\bm n}}+\beta_1 (L-L_*)\Big)\Big|_{\Gamma_\tau}=\Big(\frac{\p L_*}{\p r}-\beta_1 L_*\Big)\Big|_{r=1-\epsilon+\tau S} - \Big(\frac{\p L_*}{\p r}-\beta_1 L_*\Big)\Big|_{r=1-\epsilon} + O(|\tau S|^2).\label{newbdy2}
\end{gather*}
Since $L_*, H_*, F_*$ are all bounded and $|L_*'|\le C\epsilon$ by \re{deri}, we know from the equation \re{r1} that $|L_*''|$ is bounded with a bounded independent of $\epsilon$ and $\tau$. Hence we can find a constant, we denote it by $\t C$, which does not depend upon $\epsilon$ and $\tau$, such that
\begin{equation}\label{bdyLnew}
    \bigg|\Big(\frac{\p (L-L_*)}{\p {\bm n}}+\beta_1 (L-L_*)\Big)\Big|_{\Gamma_\tau}\bigg| \le \t C |\tau|\|S\|_{C^{4+\alpha}(\Sigma)}.
\end{equation}
Similarly, we can write the equations of $H-H_*$, $F-F_*$ and $p-p_*$ as
\begin{eqnarray}
    &-\Delta (H-H_*) = b_3(r)(H-H_*) + b_4(r)(F-F_*) \hspace{2em} &\text{in } \Omega_\tau,\label{hhs}\\
    &\label{ffs} \begin{split}
    -D\Delta (F-F_*) - \nabla p_*\cdot\nabla(F-F_*) =& \nabla F\cdot \nabla(p-p_*) + b_5(r)(L-L_*)\\
    &\quad+b_6(r)(H-H_*) + b_7(r)(F-F_*) 
    \end{split}\hspace{2em} &\text{in }\Omega_\tau,\\
    &-\Delta(p-p_*) = b_8(r)(L-L_*) + b_9(r)(H-H_*) + b_{10}(r)(F-F_*) \hspace{2em} &\text{in } \Omega_\tau,\label{pps}
\end{eqnarray}
where $b_i(r)$, $i=3,\cdots,10$ are all bounded, and it is shown earlier that $\|\nabla F\|_{L^\infty}$ and $\|\nabla p_*\|_{L^\infty}$ are bounded; for simplicity, we shall use the same constant $\t C$ to control $\|\nabla F\|_{L^\infty}$ and $\|\nabla p_*\|_{L^\infty}$, namely,
\be \label{pF}
    \|\nabla F\|_{L^\infty}\le \t C,\hspace{2em}  \|\nabla p_*\|_{L^\infty} \le \t C.
\ee 
Furthermore, the boundary conditions for $H-H_*$, $F-F_*$ and $p-p_*$ satisfy
\begin{eqnarray}
    &\frac{\p (H-H_*)}{\p r}\Big|_{r=1} = \frac{\p (F-F_*)}{\p r}\Big|_{r=1} = \frac{\p (p-p_*)}{\p r}\Big|_{r=1} = 0,\label{newbdy3}\\
    &\bigg|\Big(\frac{\p (H-H_*)}{\p {\bm n}}+\beta_1 (H-H_*)\Big)\Big|_{\Gamma_\tau}\bigg|\le \t C|\tau|\|S\|_{C^{4+\alpha}(\Sigma)},\label{newbdy4}\\
    &\bigg|\Big(\frac{\p (F-F_*)}{\p {\bm n}}+\beta_2 (F-F_*)\Big)\Big|_{\Gamma_\tau}\bigg|\le \t C|\tau|\|S\|_{C^{4+\alpha}(\Sigma)},\label{newbdy5}\\
    &\Big|(p-p_*)|_{\Gamma_\tau}\Big|\le \t C|\tau|\|S\|_{C^{4+\alpha}(\Sigma)},\label{newbdy6}
\end{eqnarray}
where the last inequality is based on the formula of $\kappa$ in \re{kappa}.

Since the system is highyly coupled, it is not an easy job to prove the estimates in Lemma \ref{first}. To show that, we use the idea of continuation (Appendix \ref{appcon}). We multiply the right-hand sides of \re{lls} -- \re{pps} by $\delta$ with $0\le \delta \le1$, and we shall 
combine the proofs for the case $\delta = 0 $ as well as the case $0<\delta\le 1$.

 We first assume that, in the case $\delta >0$, for some $M_1>0$
to be determined later on
 \begin{eqnarray}
     &\max\Big(\label{assum1} \|L-L_*\|_{L^\infty(\Omega_\tau)}, \|H-H_*\|_{L^\infty(\Omega_\tau)} , \|F-F_*\|_{L^\infty(\Omega_\tau)}\Big)  \le 2M_1|\tau|\|S\|_{C^{4+\alpha}(\Sigma)},\\
     \label{assum2} &\|\nabla (F-F_*)\|_{L^\infty(\Omega_\tau)} \le \frac{2M_1 C_s}{\epsilon}|\tau|\|S\|_{C^{4+\alpha}(\Sigma)},\\
    \label{assum4} &
    \dis \|p-p_*\|_{L^\infty(\Omega_\tau)} \le 3 \t C |\tau|\|S\|_{C^{4+\alpha}(\Sigma)}, \hspace{2em} \|\nabla(p-p_*)\|_{L^\infty(\Omega_\tau)} \le \frac{3 C_s \t C}{\epsilon}|\tau|\|S\|_{C^{4+\alpha}(\Sigma)},
 \end{eqnarray}
 where $\t C$ is from \re{bdyLnew}, \re{pF}, \re{newbdy4}--\re{newbdy6}, and $C_s$ is a scaling factor which comes from applying the $C^{1+\alpha}$ Schauder estimate as we did in Lemma \ref{lem3.1}; both $\t C$ and $C_s$ are independent of $\epsilon$ and $\tau$.
 
Let's first show \re{assum4}. Based on \re{assum1}, for the case $\delta >0$, the right-hand side of \re{pps} is bounded, i.e.,
\begin{equation}\label{pp*}
    |\Delta (p-p_*)| \le 2M_1\delta (\|b_8\|_{L^\infty(\Omega_\tau)}+\|b_9\|_{L^\infty(\Omega_\tau)} + \|b_{10}\|_{L^\infty(\Omega_\tau)}) |\tau| \|S\|_{C^{4+\alpha}(\Sigma)};
\end{equation}
notice that the above estimates is automatically valid 
in the case $\delta = 0$ without the 
assumptions \re{assum1} since the right-hand side is zero.
Let 
\begin{equation}\label{cos1}
    \phi_1(r)=2\tilde{C}|\tau|\|S\|_{C^{4+\alpha}(\Sigma)}\cos\Big(\frac{1-r}{\epsilon}\Big),
\end{equation}
then
\begin{equation*}
    \phi_1'(r) = \frac{2\t C}{\epsilon}\sin\Big(\frac{1-r}{\epsilon}\Big)|\tau|\|S\|_{C^{4+\alpha}(\Sigma)}, \hspace{2em} \phi_1''(r)=-\frac{2\t C}{\epsilon^2}\cos\Big(\frac{1-r}{\epsilon}\Big)|\tau|\|S\|_{C^{4+\alpha}(\Sigma)},
\end{equation*}
and
\begin{equation*}
    -\Delta \phi_1 = \Big[\frac{1}{\epsilon}\cos\Big(\frac{1-r}{\epsilon}\Big) - \sin\Big(\frac{1-r}{\epsilon}\Big)\Big]\frac{2\t C}{\epsilon}|\tau|\|S\|_{C^{4+\alpha}(\Sigma)},\hspace{2em} \phi_1\Big|_{\Gamma_\tau} = 2\t C\cos\Big(1-\frac{\tau S}{\epsilon}\Big)|\tau|\|S\|_{C^{4+\alpha}(\Sigma)},
\end{equation*}
where $\t C$, again, comes from \re{bdyLnew}, \re{pF}, \re{newbdy4}--\re{newbdy6}. Notice that $\cos 1\approx 0.54>1/2$, and we use the
smallness of $\epsilon$ to majorize the right-hand side of \re{pp*} for a supersolution for $p-p_*$ when $\epsilon$ is small and $|\tau|\ll \epsilon$, hence
\begin{equation*}
    \|p-p_*\|_{L^\infty(\Omega_\tau)} \le 2\tilde{C}|\tau|\|S\|_{C^{4+\alpha}(\Sigma)}.
\end{equation*}
Using a scaling argument as in \re{p-reg}, we further have
\begin{equation}\label{nablap}
    \|\nabla(p-p_*)\|_{L^\infty(\Omega_\tau)} \le \frac{2C_s\tilde{C}}{\epsilon}|\tau|\|S\|_{C^{4+\alpha}(\Sigma)}.
\end{equation}
In the next step, let's consider $L-L_*$ and $H-H_*$. It follows from the assumption \re{assum1} that
\begin{equation}
\label{bound1} |\Delta (L-L_*)| \le C M_1 \delta |\tau|\|S\|_{C^{4+\alpha}(\Sigma)}, \mm
 |\Delta (H-H_*)| \le C M_1 \delta |\tau|\|S\|_{C^{4+\alpha}(\Sigma)},
\end{equation}
where $C$ is some universal constant. Recalling also \re{pF} and \re{nablap}, we have the following estimate for $F-F_*$,
\begin{equation}
    \begin{split}
        \label{bound2} \Big|\Delta (F-F_*) + \frac1D\nabla p_*\cdot\nabla(F-F_*)\Big| \;\le&\;\;  \Big\|\frac1D \nabla F\cdot \nabla (p-p_*)\Big\|_{L^\infty} + \Big\| \frac{b_{5}(r)}{D}(L-L_*)\Big\|_{L^\infty}\\
        &\hspace{4em}+ \Big\|\frac{b_{6}(r)}{D}(H-H_*)\Big\|_{L^\infty} + \Big\|\frac{b_{7}(r)}{D}(F-F_*)\Big\|_{L^\infty}\\
 \;\le&\;\; \Big(\frac{2C_s}{\epsilon D}  \t C^2 + C M_1\Big) \delta |\tau|\|S\|_{C^{4+\alpha}(\Sigma)}.
    \end{split}
\end{equation}
We use
\begin{equation}
    \label{cos2}
    \phi_2(r)=M_1 |\tau|\|S\|_{C^{4+\alpha}(\Sigma)}\cos\Big(\frac{M_2(1-r)}{\sqrt{\epsilon}}\Big), \mm M_2 = \frac12 \min\Big(\sqrt{\beta_1},\sqrt{\beta_2}\Big),
\end{equation}
as the supersolution with $M_1$ given by
 \be  \label{m1}
  M_1 =  \max\Big(\frac{8}{\beta_1}\t C, \; \frac{8}{\beta_2}\t C,\;
  \frac{32 C_s}{\beta_1 D} \t C^2,\; \frac{32 C_s}{\beta_2 D} \t C^2 \Big)  .
  \ee 
Taking derivatives of $\phi_2$ gives us
\begin{eqnarray*}
    &\phi_2'(r) = M_1 \frac{M_2}{\sqrt{\epsilon}} |\tau| \|S\|_{C^{4+\alpha}(\Sigma)} \sin\Big(\frac{M_2(1-r)}{\sqrt{\epsilon}}\Big),\mm
    \phi_2''(r) = -M_1 \frac{M_2^2}{\epsilon} |\tau| \|S\|_{C^{4+\alpha}(\Sigma)} \cos\Big(\frac{M_2(1-r)}{\sqrt{\epsilon}}\Big).
\end{eqnarray*}
It is clear that $\phi_2'(1)=0$. Moreover, for the boundary condition at $\Gamma_\tau: r=1-\epsilon+\tau S$,
\begin{equation*}
\begin{split}
    \Big(\frac{\p\phi_2}{\p{\bm n}} + \beta_1 \phi_2\Big)\Big|_{\Gamma_\tau} &= -\phi_2'(1-\epsilon+\tau S) + \beta_1 \phi_2(1-\epsilon+\tau S) + O(|\tau S'|^2)\\
    &= \Big[- \frac{M_2}{\sqrt{\epsilon}}\sin\Big(\frac{M_2(\epsilon-\tau S)}{\sqrt{\epsilon}}\Big) + \beta_1 \cos\Big(\frac{M_2(\epsilon-\tau S)}{\sqrt{\epsilon}}\Big) \Big] M_1 |\tau| \|S\|_{C^{4+\alpha}(\Sigma)} + O(|\tau S'|^2).
    \end{split}
\end{equation*}
Since $\sin x \le x$ and $\cos x\ge 1 - \frac{x^2}2$ for $x\ge 0$,
we have, for $0<|\tau|\ll \epsilon$ and $\epsilon$ small,
\[
\frac{M_2}{\sqrt{\epsilon}}\sin\Big(\frac{M_2(\epsilon-\tau S)}{\sqrt{\epsilon}}\Big)\le M_2^2\Big(1-\frac\tau\epsilon S\Big)
\le 2 M_2^2, \mm\cos\Big(\frac{M_2(\epsilon-\tau S)}{\sqrt{\epsilon}}\Big) \ge 1 - \frac{M_2^2}{2\epsilon} (\epsilon^2 +  \tau^2  ) \ge \frac34.
\]
Then 
\begin{equation*}
    \begin{split}
        \Big(\frac{\p\phi_2}{\p{\bm n}} + \beta_1 \phi_2\Big)\Big|_{\Gamma_\tau} \;&\ge\;\; \Big[-2M_2^2 + \frac34 \beta_1\Big]M_1 |\tau|\|S\|_{C^{4+\alpha}(\Sigma)} + O(|\tau S'|^2) \\
        &\ge\;\; \frac14\beta_1 M_1 |\tau|\|S\|_{C^{4+\alpha}(\Sigma)} + O(|\tau S'|^2)\\
        &\ge\;\; 2\t C |\tau|\|S\|_{C^{4+\alpha}(\Sigma)} + O(|\tau S'|^2)\ge \tilde{C}|\tau|\|S\|_{C^{4+\alpha}(\Sigma)}.
    \end{split}
\end{equation*}
Next we consider the equations \re{lls}, \re{hhs}, \re{ffs} in proving $\phi_2$ is a supersolution. Notice that
\begin{equation*}
    \begin{split}
        -\Delta \phi_2 = -\phi_2''(r) -\frac1r \phi_2'(r) \;&=\;\; M_1\Big[\frac{M_2^2}{\epsilon}\cos\Big(\frac{M_2(1-r)}{\sqrt{\epsilon}}\Big) - \frac{M_2}{\sqrt{\epsilon}r}\sin\Big(\frac{M_2(1-r)}{\sqrt{\epsilon}}\Big)\Big] |\tau| \|S\|_{C^{4+\alpha}(\Sigma)}\\
        &\ge\;\; M_1\Big[\frac{M_2^2}{\epsilon}\frac34 - \frac{2M_2^2}{r}  \Big] |\tau| \|S\|_{C^{4+\alpha}(\Sigma)}\\
      & \ge\;\; M_1\Big[\frac{M_2^2}{\epsilon}\frac34 - 4 M_2^2  \Big] |\tau| \|S\|_{C^{4+\alpha}(\Sigma)}, \m r\in [1-\epsilon+\tau S,1] ,
    \end{split}
\end{equation*}
For \re{bound1}, it is clear that $-\Delta\phi_2 \ge \max\{|\Delta (L-L_*)|, |\Delta(H-H_*)|\}$ since the leading order term in $-\Delta \phi_2$ is $\frac{1}{\epsilon}$ and we can take $\epsilon$ small. Hence $\phi_2$ is a supersolution for $L-L_*$ as well as for $H-H_*$. For \re{bound2}, as is shown, the leading order term in bounding \re{bound2} is $\frac{2C_s\tilde{C}^2}{\epsilon D}|\tau|\|S\|_{C^{4+\alpha}(\Sigma)}$; on the other hand, 
$$-\Delta \phi_2 \ge M_1\Big[\frac{M_2^2}{\epsilon}\frac34 - \frac{2M_2^2}{r}  \Big] |\tau| \|S\|_{C^{4+\alpha}(\Sigma)}\ge \frac{1}{\epsilon}\frac12 M_1 M_2^2|\tau| \|S\|_{C^{4+\alpha}(\Sigma)} \ge \frac{4C_s\tilde{C}^2}{\epsilon D}|\tau|\|S\|_{C^{4+\alpha}(\Sigma)}; $$
the extra term $\frac1D \nabla p_*\cdot\nabla \phi_2$ is of order
$O(1/\sqrt\epsilon)$ and therefore does not cause a problem.
Thus $\phi_2$ is also a supersolution for $F-F_*$. 

From the above analysis, we see that the choice of $M_1$ and $M_2$ depends only on $\beta_1$, $\beta_2$, $\tilde{C}$ and $C_s$, and is therefore independent of $\epsilon$ and $\tau$. By the maximum principle,
\begin{eqnarray*}
    \|L-L_*\|_{L^\infty(\Omega_\tau)} &\le& M_1|\tau|\|S\|_{C^{4+\alpha}(\Sigma)},\\
    \|H-H_*\|_{L^\infty(\Omega_\tau)} &\le& M_1|\tau|\|S\|_{C^{4+\alpha}(\Sigma)},\\
    \|F-F_*\|_{L^\infty(\Omega_\tau)} &\le& M_1|\tau|\|S\|_{C^{4+\alpha}(\Sigma)}.
\end{eqnarray*}
Using a scaling argument, we further have
\begin{equation*}
    \|\nabla(F-F_*)\|_{L^\infty(\Omega_\tau)} \;\le\; \frac{M_1 C_s}{\epsilon}|\tau|\|S\|_{C^{4+\alpha}(\Sigma)}.
\end{equation*}
These estimates are valid in the case $\delta =0$ without the assumptions 
\re{assum1}--\re{assum4} since the right-hand sides are all zero in this case.
Conditions (i) and (ii) of Lemma \ref{con} are therefore
satisfied for the vectors
$\Big\{ \frac1{M_1}\|L-L_*\|_{L^\infty},
\frac1{M_1}\|H-H_*\|_{L^\infty}, \frac1{M_1}\|F-F_*\|_{L^\infty}, \frac{\epsilon}{M_1 C_s}\|\nabla(F-F_*)\|_{L^\infty},
\frac1{2\t C}\|p-p_*\|_{L^\infty},
\frac\epsilon{2 C_s\t C}\|\nabla (p- p_*)\|_{L^\infty},
\Big\}$. Since condition (iii) is obvious, we finish the proof. 
\end{proof}

\begin{rem}\label{remfirst}
Based on Lemma \ref{first}, if we further apply the Schauder estimates on the equations for $L-L_*$, $H-H_*$, $F-F_*$, and $p-p_*$, we can actually obtain
$$\|L-L_*\|_{C^{4+\alpha}(\bar{\Omega}_\tau)} + \|H-H_*\|_{C^{4+\alpha}(\bar{\Omega}_\tau)} + \|F-F_*\|_{C^{4+\alpha}(\bar{\Omega}_\tau)} + \|p-p_*\|_{C^{2+\alpha}(\bar{\Omega}_\tau)} \le C|\tau|\|S\|_{C^{4+\alpha}(\Sigma)},$$
where $C$ is independent of $\tau$, but is dependent upon $\epsilon$.
\end{rem}

\subsection{Computation of $L_1$, $H_1$, $F_1$ and $p_1$} 
{In general, if $f(y)$ ($y\in R^N$, $f\in R^M$)  is a $C^2$ function with bounded second order 
derivatives, then we have the Taylor's expansion:
\be \label{lin}
\begin{array}{rcl}
   f(y) - f(y_*) & = & \dis \int_0^1 \frac{d}{dt} f\big(t y + (1-t) y_*\big) dt 
   \; = \;  \Big(\int_0^1 \nabla f\big(t y + (1-t) y_*\big) dt\Big) \cdot (y-y_*)\\
    & = &  \nabla f(y_*)\cdot  (y-y_*) + R,
 \end{array}
\ee
where the remainder $R$, given by 
$
 R  = \dis   \int_0^1 \Big(\nabla f \big(t y + (1-t) y_*\big)    -\nabla f \big(  y_*\big)  \Big) dt \cdot (y-y_*),
$
satisfies
\be \label{R0}
 |R| \le \int_0^1\|D^2 f\|_{L^\infty} |y-y_*|\; tdt \cdot |y-y_*| = \frac12\|D^2 f\|_{L^\infty} |y-y_*|^2.
\ee
Thus we have:
\begin{lem} \label{lin5}
Suppose $\mathscript{P}$ is a linear operator, $\mathscript{P}[y] = f(y)$, $\mathscript{P}[y_*] = f(y_*)$. Let $y_1$ be the linearized solution, i.e.,
 $\mathscript{P}[y_1] =  \nabla f(y_*) \cdot y_1 $. Then 
\be 
 \mathscript{P}[y-y_*-\tau y_1] = 
 \nabla f(y_*)\cdot  (y-y_* -\tau y_1) + R 
 ,
\ee
where by \re{lin},
\be \label{estimate0}
 |R| \;\le\; 
 \frac12\|D^2 f\|_{L^\infty} |y-y_*|^2.
\ee
\end{lem}


Later on we shall apply this formula with $y = (L,H,F,p)$ and
$y_*=(L_*,H_*,F_*, p_*)$.
Notice that by Lemma \ref{first}, $|y-y_*| = O(\tau)$, thus we already have $|y-y_*|^2 = O(\tau^2)$. In what follows, we only need to produce correction terms for 
the linear part of the system, i.e., 
we shall compute the functions for $L_1$, $H_1$, $F_1$ and $p_1$. Substituting \re{expand1} -- \re{expand4} into \re{b1} -- \re{b7}, and dropping higher order terms of $\tau$, we obtain the linearized system. This is equivalent to taking total differential of the right-hand side $f$ with respect to $L, H, F$ and $p$. If we write $f =(f^L, f^H, f^F, f^p)^T$, then, from \re{e1}, 
$f^L(L, H, F, p) = - k_1\frac{(M_0-F)L}{K_1+L} - \rho_1 L$, so that
\[
 \nabla f^L(L_*, H_*, F_*, p_*)\cdot (L_1, H_1, F_1, p_1)
 = \mbox{ $ -k_1\frac{(M_0-F_*)K_1L_1}{(K_1+L_*)^2} + k_1\frac{ L_* F_1}{K_1+L_*} -\rho_1L_1,$}
\]
and this is the right-hand side of the equation for $L_1$.
Similar equations are derived for $H_1$ and $p_1$. The right-hand side 
for $F_1$ is similar, but we have to take care of the additional gradient 
terms in the left-hand side. In summary, we obtain the following linearized system on $\Omega_*=\{1-\epsilon<r<1\}$:
\begin{eqnarray}
    &-\Delta L_1 = 
    -k_1\frac{(M_0-F_*)K_1L_1}{(K_1+L_*)^2} + k_1\frac{L_* F_1}{K_1+L_*} -\rho_1L_1
    \hspace{2em} &\text{in }\Omega_*,\label{L1}\\
    &-\Delta H_1 
     = -k_2 \frac{K_2 H_* F_1}{(K_2+F_*)^2}-k_2\frac{ F_* H_1}{K_2+F_*}-\rho_2 H_1 
    \hspace{2em} &\text{in }\Omega_*,\label{H1}\\
    &-D\Delta F_1 - \nabla F_1\cdot \nabla p_*-\nabla F_*\cdot \nabla p_1= 
    k_1 \frac{(M_0-F_*)K_1L_1}{(K_1+L_*)^2}
    - k_1\frac{F_1L_*}{K_1+L_*} 
    +\cdots 
   \hspace{2em} &\text{in }\Omega_*,\label{F1}\\
    &-\Delta p_1=\frac{1}{M_0}\Big[\lambda\frac{(M_0-F_*)L_1}{\gamma+H_*} -\lambda\frac{L_*F_1}{\gamma+H_*} -\lambda\frac{(M_0-F_*)L_*H_1}{(\gamma+H_*)^2} + (\rho_3-\rho_4)F_1\Big] \hspace{2em} &\text{in }\Omega_*,\label{eqnp1}\\
    &\frac{\p L_1}{\p r}=\frac{\p H_1}{\p r}=\frac{\p F_1}{\p r}=\frac{\p p_1}{\p r}=0 \hspace{2em}& r=1,\label{bdy1}\\
    &-\frac{\p L_1}{\p r}+\beta_1 L_1=\Big(\frac{\p^2 L_*}{\p r^2}-\beta_1\frac{\p L_*}{\p r}\Big)\Big|_{r=1-\epsilon} S(\theta) &r=1-\epsilon,\label{bdyL1}\\
    &-\frac{\p H_1}{\p r}+\beta_1 H_1=\Big(\frac{\p^2 H_*}{\p r^2}-\beta_1\frac{\p H_*}{\p r}\Big)\Big|_{r=1-\epsilon} S(\theta) &r=1-\epsilon,\label{bdyH1}\\
    &-\frac{\p F_1}{\p r}+\beta_2 F_1=\Big(\frac{\p^2 F_*}{\p r^2}-\beta_2\frac{\p F_*}{\p r}\Big)\Big|_{r=1-\epsilon} S(\theta) &r=1-\epsilon,\label{bdyF1}\\
    &p_1 = \frac{1}{(1-\epsilon)^2}(S+S_{\theta\theta}) &r=1-\epsilon. \label{bdyp1}
\end{eqnarray}
\void{
where
\begin{equation}\label{complex}
    f_3 = k_1 \frac{(M_0-F_*)K_1L_1}{(K_1+L_*)^2}
    - k_1\frac{F_1L_*}{K_1+L_*} 
    +\cdots 
\end{equation}
is a complex combination of $L_1, H_1,$ and $F_1$; the first three terms listed in \re{complex} all come from the linearization of $k_1\frac{(M_0-F)L}{K_1+L}$.
}

Using the same techniques as in the proof of Lemma \ref{first}, also recalling Remark \ref{remfirst}, we can derive $L_1, H_1, F_1 \in C^{4+\alpha}(\bar{\Omega}_\tau)$ and $p_1\in C^{2+\alpha}(\bar{\Omega}_\tau)$; their Schauder estimates may depend on $\epsilon$, {\em but it is crucial that the $L^\infty$ estimates are independent of $\epsilon$ and $\tau$.}

Notice that $L_1$, $H_1$, $F_1$ and $p_1$ are all defined in $\Omega_*$, while $L-L_*$, $H-H_*$, $F-F_*$ and $p-p_*$ are defined in $\Omega_\tau$. We would now like to transform $L_1$, $H_1$, $F_1$ and $p_1$ from $\Omega_*$ to $\Omega_\tau$ so that we are able to work on the same domain to derive second-order $\tau$ estimates.
To do that, we define a transform
\begin{equation}\label{Ytau}
    Y_\tau:\; (r,\theta)=Y_\tau(\bar{r},\bar{\theta})=\Big( \frac{(\bar{r}-1)(\epsilon-\tau S)}{\epsilon} + 1\,,\,\theta\Big)
\end{equation}
and let 
\begin{eqnarray}
    \label{trans1}&\bar{L}_1(r,\theta)=L_1(Y_\tau^{-1}(r,\theta)), \hspace{2em} &\bar{H}_1(r,\theta)=H_1(Y_\tau^{-1}(r,\theta)),\\
    \label{trans2}&\bar{F}_1(r,\theta)=F_1(Y_\tau^{-1}(r,\theta)), \hspace{2em} &\bar{p}_1(r,\theta)=p_1(Y_\tau^{-1}(r,\theta)),
\end{eqnarray}
for $(r,\theta)\in \Omega_\tau$. Similar as using the Hanzawa transformation in \cite{CE3, angio1, Fengjie, Hongjing, Zejia, zhao2}, the error incurred from applying $Y_\tau$ is less than $|\tau S|$.

\subsection{Second-order $\tau$ estimates}
The first step in deriving second-order $\tau$ estimates is to calculate the equations for $L-L_*-\tau \bar{L}_1$, $H-H_*-\tau \bar{H}_1$, $F-F_*-\tau \bar{F}_1$ and $p-p_*-\tau \bar{p}_1$. 
Here we shall only show the derivations of the equation for $F-F_*-\tau \bar{F}_1$, since the equation for $F$ is more complex than those for other variables.

Combining the equations for $F_*, F,$ and $F_1$ respectively in \re{r3} \re{b3} and \re{F1}, we derive
\begin{equation}\label{fff}
    -D \Delta(F-F_*-\tau F_1) - \nabla F\cdot \nabla p + \nabla F_*\cdot \nabla p_* + \tau \nabla F_1 \cdot \nabla p_* + \tau \nabla F_* \cdot \nabla p_1 = \text{RHS}.
\end{equation}
By Lemma \ref{lin5}, the right-hand side of \re{fff} satisfies
\begin{equation*}
    \text{RHS} = [\text{\Rmnum{1}}] + [\text{\Rmnum{2}}],
\end{equation*}
where \Rmnum{1} is written as, for bounded functions $b_{11}(r), b_{12}(r),$ and $b_{13}(r)$,
\begin{equation*}
    \text{\Rmnum{1}} = b_{11}(r)(L-L_*-\tau L_1) + b_{12}(r)(H-H_*-\tau H_1) + b_{13}(r)(F-F_*-\tau F_1);
\end{equation*}
and \Rmnum{2} is bounded by  $|(L-L_*,H-H_*,F-F_*)|^2$, hence
\begin{equation*}
    \|\text{\Rmnum{2}}\|_{L^\infty} \le C|\tau|^2 \|S\|_{C^{4+\alpha}(\Sigma)}.
\end{equation*}

\void{
\begin{equation*}
    \text{RHS} = k_1\frac{(M_0-F)L}{K_1+L}
     - k_1\frac{(M_0-F_*)L_*}{K_1+L_*} - \tau\Big[k_1 \frac{(M_0-F_*)K_1L_1}{(K_1+L_*)^2}
    - k_1\frac{F_1L_*}{K_1+L_*}\Big] + \cdots.
\end{equation*}
}

We then turn to the left-hand side of equation \re{fff}. The terms involving the gradients can be rearranged as
\begin{equation*}
    \begin{split}
        &- \nabla F\cdot \nabla p + \nabla F_*\cdot \nabla p_* + \tau \nabla F_1 \cdot \nabla p_* + \tau \nabla F_* \cdot \nabla p_1\\
        =& -\nabla p_* \cdot \nabla(F-F_*-\tau F_1) - \nabla F\cdot \nabla(p-p_*-\tau p_1) - \tau \nabla(F-F_*)\cdot \nabla p_1.
    \end{split}
\end{equation*}  
By Lemma \ref{first}, 
\begin{equation}
\label{last}    \|\nabla (F-F_*)\|_{L^\infty} \le \frac{C}{\epsilon}|\tau|\|S\|_{C^{4+\alpha}(\Sigma)};
\end{equation}
furthermore, we can derive from \re{eqnp1} and \re{bdyp1} that
\begin{equation*}
    |\Delta (p_1 - (S+S_{\theta \theta}))| \le C, \quad \text{and} \quad \|p_1-(S+S_{\theta\theta})\|_{C^{1+\alpha}(\{r=1-\epsilon\})} \le C\epsilon,
\end{equation*}
as $S\in C^{4+\alpha}$; using the same technique as in Lemma \ref{lem3.1}, we shall get
\begin{equation*}
    \|\nabla (p_1 - (S+S_{\theta\theta}))\|_{L^\infty(\Omega_*)} \le C,
\end{equation*}
hence
\begin{equation*}
    \|\nabla p_1\|_{L^\infty(\Omega_*)} \le C,
\end{equation*}
for a constant $C$ which is independent of $\epsilon$ and $\tau$. 
Together with \re{last}, we derive 
\begin{equation*}
    \|\tau \nabla(F-F_*) \cdot \nabla p_1\|_{L^\infty} \le \frac{C}{\epsilon}|\tau|^2 \|S\|_{C^{4+\alpha}(\Sigma)}.
\end{equation*}

From the above analysis, we obtain the equation for $F-F_*-\tau F_1$,
\begin{equation*}
\begin{split}
    &-D\Delta (F-F_*-\tau F_1) - \nabla p_*\cdot \nabla (F-F_*-\tau F_1) \\
    &\hspace{10em}= \nabla F\cdot \nabla(p-p_*-\tau p_1) - \tau \nabla(F-F_*)\cdot \nabla p_1 + [\text{\Rmnum{1}}] + [\text{\Rmnum{2}}].
    \end{split}
\end{equation*}
Now we recall the transform $Y_\tau$ in \re{Ytau} and the change of variables in \re{trans1} and \re{trans2}, we can derive the equation for $F-F_*-\tau \bar{F}_1$, namely,
\begin{equation}\label{ffbf}
\begin{split}
    &-D\Delta (F-F_*-\tau \bar{F}_1) - \nabla p_*\cdot \nabla (F-F_*-\tau \bar{F}_1) \\
    &\hspace{10em}= \nabla F\cdot \nabla(p-p_*-\tau \bar{p}_1) - \tau \nabla(F-F_*)\cdot \nabla \bar{p}_1 + [\text{\Rmnum{1}}] + [\text{\Rmnum{2}}] + \tau f_4,
    \end{split}
\end{equation}
where $f_4$ is generated by the tiny changing of domain from $\Omega_*$ to $\Omega_\tau$ in applying the transformation $Y_\tau$, and it contains at most second derivatives of $\tau S$, hence
$$\|\tau f_4\|_{L^\infty(\Omega_\tau)}\le |\tau|\cdot C|\tau|\|S\|_{C^{2+\alpha}(\Omega_\tau)} \le C|\tau|^2 \|S\|_{C^{4+\alpha}(\Omega_\tau)}.$$
Combining with the estimates we derived before, we have
\begin{equation*}
    \begin{split}
        &\Big| \Delta(F-F_*-\tau \bar{F}_1) + \frac{1}{D}\nabla p_* \cdot \nabla (F-F_*-\tau \bar{F}_1)\Big|\; \le \; \Big\|\frac1D \nabla F \cdot \nabla (p-p_*-\tau \bar{p}_1)\Big\|_{L^\infty} \\
        &\hspace{15em}+ \Big\|\frac{b_{11}(r)}{D}(L-L_*-\tau \bar{L}_1)\Big\|_{L^\infty} + \Big\|\frac{b_{12}(r)}{D}(H-H_*-\tau \bar{H}_1)\Big\|_{L^\infty} \\
        &\hspace{15em}+\Big\|\frac{b_{13}(r)}{D}(F-F_*-\tau \bar{F}_1)\Big\|_{L^\infty} + \frac{C}{\epsilon} |\tau|^2 \|S\|_{C^{4+\alpha}(\Sigma)}.
    \end{split}
\end{equation*}
Notice that the above inequality present similar structure as \re{bound2}, hence we can use the same technique and similar supersolutions to establish
\begin{lem}\label{second1}
  Fix $\epsilon$ sufficiently small, if $|\tau|\ll \epsilon$ and $\|S\|_{C^{4+\alpha}(\Sigma)}\le 1$, then we have
 \begin{eqnarray*}
     &\max\{\|L-L_*-\tau \bar{L}_1\|_{L^\infty(\Omega_\tau)} , \|H-H_*-\tau \bar{H}_1\|_{L^\infty(\Omega_\tau)}\}\;\le\; C|\tau|^2 \|S\|_{C^{4+\alpha}(\Sigma)},\\
     &\max\{\|F-F_*-\tau \bar{F}_1\|_{L^\infty(\Omega_\tau)} , \|p-p_*-\tau \bar{p}_1\|_{L^\infty(\Omega_\tau)}\}\;\le\; C|\tau|^2 \|S\|_{C^{4+\alpha}(\Sigma)},\\
     & \max\{ \|\nabla(F-F_*-\tau \bar{F}_1)\|_{L^\infty(\Omega_\tau)}, \|\nabla(p-p_*-\tau \bar{p}_1)\|_{L^\infty(\Omega_\tau)}\}\;\le\; \frac{C}\epsilon|\tau|^2\|S\|_{C^{4+\alpha}(\Sigma)},
 \end{eqnarray*}
 where $C$ is independent of $\epsilon$ and $\tau$.
\end{lem}
\noindent Following Remark \ref{remfirst}, we shall further have
\begin{lem}\label{second}
  Fix $\epsilon$ sufficiently small, if $|\tau|\ll \epsilon$ and $\|S\|_{C^{4+\alpha}(\Sigma)}\le 1$, then
 \begin{eqnarray}
     \label{lem33-1} &\|L-L_*-\tau \bar{L}_1\|_{C^{4+\alpha}(\bar{\Omega}_\tau)} \le C|\tau|^2\|S\|_{C^{4+\alpha}(\Sigma)},\\
     \label{lem33-2}&\|H-H_*-\tau\bar{H}_1\|_{C^{4+\alpha}(\bar{\Omega}_\tau)} \le C|\tau|^2\|S\|_{C^{4+\alpha}(\Sigma)},\\
    \label{lem33-3}&\|F-F_*-\tau\bar{F}_1\|_{C^{4+\alpha}(\bar{\Omega}_\tau)} \le C|\tau|^2\|S\|_{C^{4+\alpha}(\Sigma)},\\
    \label{lem33-4} &\|p-p_*-\tau\bar{p}_1\|_{C^{2+\alpha}(\bar{\Omega}_\tau)} \le C|\tau|^2\|S\|_{C^{4+\alpha}(\Sigma)},
 \end{eqnarray}
 where $C$ is independent of $\tau$, but is dependent on $\epsilon$.
\end{lem}

The estimates \re{lem33-1} -- \re{lem33-4} are uniformly valid for $|\tau|$ small and $\|S \|_{C^{4+\alpha}(\Sigma)}\le 1$ . By now, we finish the mathematical justification of \re{expand1} -- \re{expand4}, and we are ready to derive the Fr\'echet derivatives of $\mathcal{F}$.

\subsection{Fr\'echet derivative}
Introduce the Banach spaces
\begin{gather}
    X^{l+\alpha} = \{S\in C^{l+\alpha}(\Sigma), S \text{ is $2\pi$-periodic in $\theta$}\},\nonumber\\
    \label{Banach}
    X^{l+\alpha}_1 = \text{closure of the linear space spanned by $\{\cos(n\theta),n=0,1,2,\cdots\}$ in $X^{l+\alpha}$}.
\end{gather}

It can be easily proved that the system \re{b1} -- \re{b7} is even in variable $\theta$ if we assume $S(\theta) = S(-\theta)$. 
Together with \re{lem33-4}, we know that the mapping $\mathcal{F}(\cdot,\mu): X^{l+4+\alpha}_1 \rightarrow X^{l+1+\alpha}_1$ is bounded when $l=0$, and the same argument can show that it is also true for any $l>0$. In order to apply the Crandall-Rabinowitz theorem, we need to verify the continuous differentiability of $\mathcal{F}$. As will be shown in the following lemma, the differentiablity is eventually reduced to the regularity of the corresponding PDEs, and explicit formula is not needed if we are only interested in differentiability;  therefore a similar argument shows that 
this mapping is Fr\'echet differentiable in $(\tilde{R},\mu)$; furthermore $\partial \mathcal{F}(\tilde{R},\mu)/\partial\tilde{R}$ (or $\partial \mathcal{F}(\tilde{R},\mu)/\partial\mu$) is obtained
by solving a linearized problem about $(\tilde{R},\mu)$ with respect to $\tilde{R}$ (or $\mu$). By
using the Schauder estimates we can then further obtain differentiability of
$\mathcal{F}(\tilde{R},\mu)$ to any order.

We now proceed to compute those Fr\'echet derivatives that are crucial in applying the Crandall-Rabinowitz theorem.

\begin{lem}
  \label{FrechetD}
  The Fr\'echet derivatives of $\mathcal{F}(\tilde{R},\mu)$ at the point $(0,\mu)$ are given by
  \begin{eqnarray}
      \label{FreD}
      &\Big[\mathcal{F}_{\tilde{R}}(0,\mu)\Big]S(\theta) = \frac{\p^2 p_*}{\p r^2}\Big|_{r=1-\epsilon}S(\theta) + \frac{\p p_1}{\p r}\Big|_{r=1-\epsilon},\\
      \label{FreDD}
      &\left[\mathcal{F}_{\mu \tilde{R}}(0,\mu)\right]S(\theta) = \frac{\p}{\p \mu}\Big(\frac{\p^2 p_*}{\p r^2}\Big|_{r=1-\epsilon}\Big)S(\theta) + \frac{\p}{\p \mu}\Big(\frac{\p p_1}{\p r}\Big|_{r=1-\epsilon}\Big).
  \end{eqnarray}
\end{lem}
\begin{proof}
Since 
\begin{equation*}
    \frac{\p p_*}{\p r}\Big|_{r=1-\epsilon}=0,
\end{equation*}
which implies $\mathcal{F}(0,\mu)=0$. For $\tilde{R}=\tau S$, it then follows from \re{lem33-4} that
\begin{equation*}
    \begin{split}
        \mathcal{F}(\tau S,\mu)= -\frac{\p p}{\p {\bm n}}\Big|_{\Gamma_\tau}&=  \frac{\p (p_*+\tau p_1)}{\p r}\Big|_{r=1-\epsilon+\tau S} + O(|\tau|^2 \|S\|_{C^{4+\alpha}(\Sigma)})\\
        &=\tau\Big[\frac{\p^2 p_*}{\p r^2}\Big|_{r=1-\epsilon}S(\theta) + \frac{\p p_1}{\p r}\Big|_{r=1-\epsilon}\Big] + O(|\tau|^2 \|S\|_{C^{4+\alpha}(\Sigma)}),
    \end{split}
\end{equation*}
which leads to the expression of the Fr\'echet derivative in \re{FreD}, and \re{FreDD} is a direct consequence of \re{FreD}.
\end{proof}

\section{Bifurcations - Proof of Theorem \ref{result}}
In this section, we shall employ the explicit expression of the Fr\'echet derivative \re{FreD} to verify the four conditions in the Crandall-Rabinowitz theorem and complete the proof of Theorem \ref{result}. Unlike \cite{CE3, FFBessel, FR2, FH3, angio1, Fengjie, Hongjing, Zejia, W, WZ2}, we cannot solve $p_*$ and $p_1$ explicitly, since our model is highly nonlinear and coupled. To meet the challenges, we need to derive various sharp estimates on $p_*$ and $p_1$. 

Throughout the rest of this paper, $C$ is used to represent a generic constant independent of $\epsilon$, which might change from line to line.

\subsection{Estimates for $p_*$}
In order to estimate $\frac{\p^2 p_*(1-\epsilon)}{\p r^2}$ in \re{FreD}, we start with evaluating \re{r4} at $r=1-\epsilon$ and substituting the boundary condition $\re{r7a}$, hence we obtain
\begin{equation}\label{psproof}
    -\frac{\p^2 p_*(1-\epsilon)}{\p r^2} = \frac{1}{M_0}\Big(\lambda\frac{(M_0-F_*)L_*}{\gamma+H_*}-\rho_3(M_0-F_*) - \rho_4 F_*\Big)\Big|_{r=1-\epsilon}.
\end{equation}
Similar to the proof of Theorem \ref{thm21}, we substitute \re{Ls} -- \re{Fs} into the above formula and combine with \re{rho4}, we find that both $O(1)$ and $O(\epsilon)$ terms cancel out, thus
\begin{equation}\label{psprop}
    \frac{\p^2 p_*(1-\epsilon)}{\p r^2} = \frac{\epsilon}{M_0}\Big(\frac{M_0}{\gamma+H_0}(\mu-\mu_c)-\rho_4 F_*^1\Big) + O(\epsilon^2) = O(\epsilon^2).
\end{equation}
Denote
\begin{equation}
    \label{J1}
    J_1(\mu,\rho_4) = \frac{1}{\epsilon^2}\frac{\p^2 p_*(1-\epsilon)}{\p r^2},\quad \text{ i.e., } \quad \frac{\p^2 p_*(1-\epsilon)}{\p r^2}=\epsilon^2 J_1(\mu,\rho_4), 
\end{equation}
it follows from \re{psprop} that  $J_1(\mu,\rho_4) = O(1)$ is bounded. Besides, we claim that $\frac{\dif J_1}{\dif \mu} = \frac{\p J_1}{\p \mu}+\frac{\p J_1}{\p \rho_4}\frac{\p \rho_4}{\p \mu}= O(1)$ is also bounded. To prove it, we take $\mu$ derivative of equation \re{psprop}, and derive
\begin{equation*}
        \frac{\p^2 }{\p r^2}\Big(\frac{\p p_*}{\p \mu}\Big)\Big|_{r=1-\epsilon} =\epsilon \Big(\frac{1}{\gamma+H_0} - \frac{F_*^1}{M_0}\frac{\p \rho_4}{\p \mu}\Big) + O(\epsilon^2).
\end{equation*}
By substituting the formula of $\frac{\p \rho_4}{\p \mu}$ in \re{rho4D}, we find that the $O(\epsilon)$ terms in the above equation cancel out, hence
\begin{equation*}
        \frac{\dif J_1(\mu,\rho_4(\mu))}{\dif \mu} = \frac{1}{\epsilon^2} \frac{\p^2 }{\p r^2}\Big(\frac{\p p_*}{\p \mu}\Big)\Big|_{r=1-\epsilon} = \frac{1}{\epsilon^2} O(\epsilon^2) = O(1).
\end{equation*}

To sum up, the properties of $J_1$ are listed in the following lemma:
\begin{lem}\label{J1prop} 
For function $J_1(\mu,\rho_4)$ defined in \re{J1}, there exists a constant $C$ which is independent of $\epsilon$ such that
\begin{equation}
    \label{J1bounds}
    |J_1(\mu,\rho_4(\mu))| \le C,\hspace{2em} \Big|\frac{\dif J_1(\mu,\rho_4(\mu))}{\dif \mu}\Big| \le C.
\end{equation}
\end{lem}

\subsection{Estimates for $p_1$} 
Set the perturbation 
$$S(\theta)=\cos(n\theta),$$ 
as the set $\{\cos(n\theta)\}_{n=1}^\infty$ is clearly a basis for the Banach space $X^{l+\alpha}_1$ defined in \re{Banach}. Since the solution to \re{L1} -- \re{bdyp1} $(L_1,H_1,F_1,p_1)$ is unique, we know 
if we can find a solution $(L_1, H_1, F_1, p_1)$ of the form
\begin{eqnarray}
    &L_1=L_1^n \cos(n\theta), \hspace{2em} &H_1=H_1^n\cos(n\theta),\label{nt1}\\
    &F_1=F_1^n \cos(n\theta), \hspace{2em} &p_1=p_1^n\cos(n\theta),\label{nt2}
\end{eqnarray}
then it is the unique solution of \re{L1} -- \re{bdyp1}. Substituting \re{nt1} and \re{nt2} into \re{L1} -- \re{bdyp1}, we need to find $(L_1^n, H_1^n, F_1^n, p_1^n)$ satisfying
\begin{eqnarray}
    &-\frac{\p^2 L_1^n}{\p r^2}-\frac{1}{r}\frac{\p L_1^n}{\p r} + \frac{n^2}{r^2}L_1^n= f_5(L_1^n,H_1^n,F_1^n) \hspace{2em} &\text{in }\Omega_*,\label{L1n1}\\
    &-\frac{\p^2 H_1^n}{\p r^2}-\frac{1}{r}\frac{\p H_1^n}{\p r} + \frac{n^2}{r^2}H_1^n= f_6(L_1^n,H_1^n,F_1^n) \hspace{2em} &\text{in }\Omega_*,\label{H1n1}\\
    &-D\frac{\p^2 F_1^n}{\p r^2}-\frac{D}{r}\frac{\p F_1^n}{\p r} + \frac{D n^2}{r^2}F_1^n - \frac{\p F_1^n}{\p r} \frac{\p p_*}{\p r}= f_7(L_1^n,H_1^n, F_1^n)+\frac{\p F_*}{\p r}\frac{\p p_1^n}{\p r}
   \hspace{2em} &\text{in }\Omega_*,\label{F1n1}\\
    &-\frac{\p^2 p_1^n}{\p r^2}-\frac{1}{r}\frac{\p p_1^n}{\p r} + \frac{n^2}{r^2}p_1^n=f_8(L_1^n,H_1^n,F_1^n) \hspace{2em} &\text{in }\Omega_*,\label{p1n1}\\
    &\frac{\p L_1^n}{\p r}=\frac{\p H_1^n}{\p r}=\frac{\p F_1^n}{\p r}=\frac{\p p_1^n}{\p r}=0 & r=1,\label{bdy1n1}\\
    &-\frac{\p L_1^n}{\p r}+\beta_1 L_1^n=\Big(\frac{\p^2 L_*}{\p r^2}-\beta_1\frac{\p L_*}{\p r}\Big)\Big|_{r=1-\epsilon} &r=1-\epsilon,\label{bdyL1n1}\\
    &-\frac{\p H_1^n}{\p r}+\beta_1 H_1^n=\Big(\frac{\p^2 H_*}{\p r^2}-\beta_1\frac{\p H_*}{\p r}\Big)\Big|_{r=1-\epsilon} &r=1-\epsilon,\label{bdyH1n1}\\
    &-\frac{\p F_1^n}{\p r}+\beta_2 F_1^n=\Big(\frac{\p^2 F_*}{\p r^2}-\beta_2\frac{\p F_*}{\p r}\Big)\Big|_{r=1-\epsilon} &r=1-\epsilon,\label{bdyF1n1}\\
    &p_1^n = \frac{1-n^2}{(1-\epsilon)^2} &r=1-\epsilon,\label{bdyp1n1}
\end{eqnarray}
where by \re{L1} -- \re{eqnp1}, $f_5$, $f_6$, $f_7$, and $f_8$ can all be bounded by linear functions of $|L_1^n|$, $|H_1^n|$, and $|F_1^n|$. In particular, $f_8$ is expressed as
\begin{equation}
    \label{f9}
    f_8 = \frac{1}{M_0}\Big[\lambda\frac{(M_0-F_*)L^n_1}{\gamma+H_*} -\lambda\frac{L_*F^n_1}{\gamma+H_*} -\lambda\frac{(M_0-F_*)L_*H^n_1}{(\gamma+H_*)^2} + (\rho_3-\rho_4)F_1^n\Big],
\end{equation}
which will be used later.

Denote the operator $\mathscript L \triangleq \frac{\p^2}{\p r^2} + \frac{1}{r}\frac{\p }{\p r} + \frac{n^2}{r^2}$. For this special operator, one can easily verify the following lemmas. 
\begin{lem}\label{eqn-n}
  The general solution of ($\eta$ is a constant)
\begin{eqnarray}
 && {\mathscript L}[\psi]\triangleq- \psi''- \frac1r \psi' + \frac{n^2}{r^2} \psi  = \eta + f(r), \mm 1-\epsilon < r < 1 , \label{psi0}\\
 && \psi'(1) = 0 \label{psi1}
\end{eqnarray}
is given by 
\bea
 && \psi - \psi_1  =  \left\{ \begin{array}{ll} \dis
 A r^n + B r^{-n} + K[f](r), \quad\text{where } B =  A+\frac1n K[f]'(1)   \mm & n\neq 0, \\
  A + K[f](r) & n= 0,
\end{array}\right. 
\eea
where
\bea \label{psi}\psi_1'(1) = 0, \mm
&& \psi_1 =  \left\{ \begin{array}{ll} \dis \frac{\eta}{n^2-4}\Big(r^2-\frac2nr^n\Big)
 & n\neq 0, 2,\\
 \dis\eta\Big( \frac{1-r^2}4 + \frac12\log r\Big) \hspace{1em}& n = 0, \\
 \dis \eta\Big( \frac{r^2}8 - \frac{r^2}4\log r\Big) & n= 2, 
 \end{array}\right.  
 \eea 
 and
 \bea \label{Kf} 
&&  K[f](r)  =  
 \left\{ \begin{array}{ll}
 \dis \frac{r^n}{2n} \int_r^1  s^{-n+1} f(s) \, \dif s+
  \frac{r^{-n}}{2n} \int_{1-\epsilon}^r  s^{n+1} f(s) \, \dif s \hspace{1em}& n\neq 0,\\
 \dis\rule{0pt}{18pt}- \int_{r}^1 \Big(\log\frac{s}r\Big)\; s f(s) \, \dif s & n= 0.
\end{array}\right.
 \eea
The special solution $K[f]$ satisfies
\be 
 \label{estK} | K[f](r) | \le  \min\Big(\frac{\epsilon}{2n}, \frac1{n^2}\Big) \|f\|_{L^\infty}, \mm
  | K[f]'(r) | \le  \min\Big(\frac\epsilon2, \frac1n\Big) \|f\|_{L^\infty} , \mm n\ge 1,
\ee
and
\be 
 | K[f](r) | \le  \epsilon \|f\|_{L^\infty}, \mm
  | K[f]'(r) | \le  \epsilon \|f\|_{L^\infty} , \mm n = 0.
\ee
\end{lem}
\begin{proof} Using the expression in \re{Kf}, we clearly have, for $1-\epsilon\le r\le 1$ and $n\ge 1$,
\begin{eqnarray} \label{proofK}
 | K[f](r) | \; \le \; \|f\|_{L^\infty} \Big[
   \frac{1}{2n} \int_r^1  \Big(\frac{r}{s}\Big)^{n} s \, \dif s +
  \frac{1}{2n} \int_{1-\epsilon}^r   \Big(\frac{s}{r}\Big)^{n} s \,  \dif s \Big]
  \; \le \; \frac{\epsilon}{2n} \|f\|_{L^\infty} ,
\end{eqnarray}
We can also integrate the expression to obtain
\[
\int_r^1  \Big(\frac{r}{s}\Big)^{n} s \, \dif s +
    \int_{1-\epsilon}^r   \Big(\frac{s}{r}\Big)^{n} s \,  \dif s  
 \le \int_r^1  r^n s^{-n-1} \, \dif s +
    \int_{1-\epsilon}^r   r^{-n} s^{n-1} \,  \dif s 
    \le r^n \frac{r^{-n}}n + r^{-n}\frac{r^{n}}n = \frac2n;
\]
combining it with \re{proofK}, we deduce
\[
 | K[f](r) | \le  \min\Big(\frac{\epsilon}{2n}, \frac1{n^2}\Big) \|f\|_{L^\infty}.
\]
Furthermore, it follows from \re{Kf} that
\[
  K[f]'(r) = \frac{r^{n-1}}{2} \int_r^1  s^{-n+1} f(s) \, \dif s -
  \frac{r^{-n-1}}{2} \int_{1-\epsilon}^r  s^{n+1} f(s) \, \dif s;
\]
similarly, we shall obtain
  \beaa
  | K[f]'(r) | & \le & \|f\|_{L^\infty} \Big[
   \frac{1}{2} \int_r^1  \Big(\frac{r}{s}\Big)^{n-1} \, \dif s +
  \frac{1}{2} \int_{1-\epsilon}^r   \Big(\frac{s}{r}\Big)^{n+1} \, \dif s \Big]
  \; \le \;\min\Big(\frac\epsilon2, \frac1n\Big) \|f\|_{L^\infty}  .
\eeaa
The case $n=0$ is similar.
\end{proof}

\begin{lem}\label{eqn-n1}                                                                 If in addition to \re{psi0} and \re{psi1} we further assume $\psi(1-\epsilon) = G$,
  then, for $n\ge 1$,
 \bea 
 &&\label{A} A  =  \frac1{1+(1-\epsilon)^{2n}} \Big( (1-\epsilon)^n [G-\psi_1(1-\epsilon)]
  - (1-\epsilon)^nK[f](1-\epsilon) - \frac1n K[f]'(1)\Big), \\
&&\label{B} B  =  \frac1{1+(1-\epsilon)^{2n}} \Big( (1-\epsilon)^n[G-\psi_1(1-\epsilon)]
  - (1-\epsilon)^nK[f](1-\epsilon) +\frac{(1-\epsilon)^{2n}}n K[f]'(1)\Big),
 \eea 
 and for $n=0$,
 \be 
  A = G - \psi_1(1-\epsilon) - K[f](1-\epsilon).
 \ee
\end{lem}

\begin{lem}\label{nepsilon} For $n\ge 0$ and $0<\epsilon<1$,
  \bea 
   && 1-n\epsilon \le (1-\epsilon)^n \le 1- n\epsilon +\frac12 n^2 \epsilon^2. 
  \eea 
\end{lem}
\begin{proof} The function
$f(\epsilon) \triangleq (1-\epsilon)^n - 1 + n\epsilon$ satisfies 
$f(0) = 0$ and $f'(\epsilon) = -n (1-\epsilon)^{n-1} + n \ge 0$
for $0<\epsilon <1$, so that $f(\epsilon)\ge 0$ for $0<\epsilon <1$.

Similarly, the function $f(\epsilon) \triangleq (1-\epsilon)^n - 1 + n\epsilon -\frac12 n^2\epsilon^2 $ satisfies 
$f(0) = f'(0) = 0$ and $f''(\epsilon) = n(n-1) (1-\epsilon)^{n-2} - n^2 \le 0$
for $0<\epsilon <1$, so that $f(\epsilon)\le 0$ for $0<\epsilon <1$.
\end{proof}

In order to make the boundary conditions \re{bdyL1n} -- \re{bdyF1n} homogeneous, let's instead work with
\begin{gather}
        \tilde{L_1^n}(r) = L_1^n(r) - \frac{1}{\beta_1}\Big(\frac{\p^2 L_*}{\p r^2}-\beta_1\frac{\p L_*}{\p r}\Big)\Big|_{r=1-\epsilon},\label{defL}\\
        \tilde{H_1^n}(r) = H_1^n(r) - \frac{1}{\beta_1}\Big(\frac{\p^2 H_*}{\p r^2}-\beta_1\frac{\p H_*}{\p r}\Big)\Big|_{r=1-\epsilon},\label{defH}\\
        \tilde{F_1^n}(r) = F_1^n(r) - \frac{1}{\beta_2}\Big(\frac{\p^2 F_*}{\p r^2}-\beta_2\frac{\p F_*}{\p r}\Big)\Big|_{r=1-\epsilon}.\label{defF}
\end{gather}
Accordingly, $\tilde{L_1^n}(r)$, $\tilde{H_1^n}(r)$, $\tilde{F_1^n}(r)$ satisfy the following equations:
\begin{eqnarray}
    &-\frac{\p^2 \tilde{L_1^n}}{\p r^2}-\frac{1}{r}\frac{\p \tilde{L_1^n}}{\p r} + \frac{n^2}{r^2}\tilde{L_1^n}= \t f_5 \triangleq f_5 -\frac{n^2}{\beta_1 r^2}\Big(\frac{\p^2 L_*}{\p r^2}-\beta_1\frac{\p L_*}{\p r}\Big)\Big|_{r=1-\epsilon} \hspace{1em} &\text{in }\Omega_*,\label{L1n}\\
    &-\frac{\p^2 \tilde{H_1^n}}{\p r^2}-\frac{1}{r}\frac{\p \tilde{H_1^n}}{\p r} + \frac{n^2}{r^2}\tilde{H_1^n}=\t f_6 \triangleq f_6-\frac{n^2}{\beta_1 r^2}\Big(\frac{\p^2 H_*}{\p r^2}-\beta_1\frac{\p H_*}{\p r}\Big)\Big|_{r=1-\epsilon} \hspace{1em} &\text{in }\Omega_*,\label{H1n}\\
    &\begin{array}{ll}
    -D\frac{\p^2 \tilde{F_1^n}}{\p r^2}-\frac{D}{r}\frac{\p \tilde{F_1^n}}{\p r} + \frac{D n^2}{r^2}\tilde{F_1^n}- \frac{\p \tilde{F_1^n}}{\p r} \frac{\p p_*}{\p r} = \t f_7 \triangleq f_7 +\frac{\p F_*}{\p r}\frac{\p p_1^n}{\p r}\\
    \hspace{20.5em}- \frac{D n^2}{\beta_2 r^2}\Big(\frac{\p^2 F_*}{\p r^2}-\beta_2\frac{\p F_*}{\p r}\Big)\Big|_{r=1-\epsilon}
    \end{array} \hspace{1em} &\text{in }\Omega_*,\label{F1n}\\
    &\frac{\p \tilde{L_1^n}}{\p r}=\frac{\p \tilde{H_1^n}}{\p r}=\frac{\p \tilde{F_1^n}}{\p r}=0 & r=1,\label{bdy1n}\\
    &-\frac{\p \tilde{L_1^n}}{\p r}+\beta_1 \tilde{L_1^n}=0 &r=1-\epsilon,\label{bdyL1n}\\
    &-\frac{\p \tilde{H_1^n}}{\p r}+\beta_1 \tilde{H_1^n}=0 &r=1-\epsilon,\label{bdyH1n}\\
    &-\frac{\p \tilde{F_1^n}}{\p r}+\beta_2 \tilde{F_1^n}=0 &r=1-\epsilon,\label{bdyF1n}
\end{eqnarray}
where $p_1^n$ is defined by \re{p1n1} and \re{bdyp1n1}.
 
\begin{lem}\label{lem4.1a}
For sufficiently small $\epsilon$, there exist a constant $C$ which does not depend on $\epsilon$ and $n$ such that the following inequalities are valid for the above system,
  \begin{eqnarray}
      \label{est1a}&\|\tilde{L_1^n}\|_{L^\infty(1-\epsilon,1)} +\|\tilde{H_1^n}\|_{L^\infty(1-\epsilon,1)} +  \|\tilde{F_1^n}\|_{L^\infty(1-\epsilon,1)} \le C(n^2+1) \epsilon,\\
  & \label{est2a} \| ({p_1^n})'\|_{L^\infty(1-\epsilon,1)} \le 2 (n^3+1).
  \end{eqnarray}
\end{lem}
 \begin{proof}
To prove \re{est1a} and \re{est2a}, we again use the idea of continuation (Appendix~\ref{appcon}), and multiply the right-hand sides of \re{L1n} -- \re{F1n} as well as \re{p1n1} by $\delta$. When $\delta=0$, it follows from the maximum principle that $\tilde{L_1^n}=\tilde{H_1^n}=\tilde{F_1^n}=0$, hence \re{est1a} clearly holds in this case. Furthermore, it can be solved from 
\begin{eqnarray}
    \label{eqn1a}&& -\frac{\p^2 p_1^n}{\p r^2} -\frac1r \frac{\p p_1^n}{\p r} + \frac{n^2}{r^2}p_1^n = 0 \mm 1-\epsilon < r < 1,\\
    \label{eqn2a}&&\frac{\p p_1^n(1)}{\p r} = 0, \mm p_1^n(1-\epsilon)=\frac{1-n^2}{(1-\epsilon)^2},
\end{eqnarray} 
that 
\begin{equation*}
    p_1^n(r) = \frac{1-n^2}{(1-\epsilon)^2[(1-\epsilon)^n + (1-\epsilon)^{-n}]}\Big(r^n + r^{-n}\Big),
\end{equation*}
and hence for $0<\epsilon\ll 1$,
\begin{equation*}
\begin{split}
    \|(p_1^n)'\|_{L^\infty(1-\epsilon,1)} &= \max\limits_{1-\epsilon\le r\le 1} \bigg|\frac{1-n^2}{(1-\epsilon)^2[(1-\epsilon)^n + (1-\epsilon)^{-n}]}n\Big(r^{n-1}-r^{-n-1}\Big)\bigg|\\
    &\le n(n^2-1)\bigg|\frac{1}{(1-\epsilon)^3}\frac{(1-\epsilon)^n-(1-\epsilon)^{-n}}{(1-\epsilon)^n+(1-\epsilon)^{-n}}\bigg| \le 2(n^3 + 1).
    \end{split}
\end{equation*}

Next we consider the case when $0< \delta \le 1$. We first assume that
\begin{eqnarray}
      \label{1a}& \|\tilde{L_1^n}\|_{L^\infty(1-\epsilon,1)} +\|\tilde{H_1^n}\|_{L^\infty(1-\epsilon,1)} +\|\tilde{F_1^n}\|_{L^\infty(1-\epsilon,1)} \le  n^2+1,\\
     \label{2a} & \| ({p_1^n})'\|_{L^\infty(1-\epsilon,1)} \le 3(n^3+1).
  \end{eqnarray}
Then clearly $|\t f_5|\le C (n^2+1)$, and $\Big(K[\t f_5](r) +
\Big|K[\t f_5]'(1)\Big|\Big(r+\frac1{\beta_1}\Big)+ \frac1{\beta_1}\Big|K[\t f_5]'(1-\epsilon)-{\beta_1} K[\t f_5](1-\epsilon)\Big|\Big)$ is a  supersolution for $\tilde{L_1^n}(r)$ when $n\ge 1$. It follows that, by Lemma \ref{eqn-n},
  \[
  \Big|\tilde{L_1^n}(r)\Big| \le K[\t f_5](r)+ \Big|K[\t f_5]'(1)\Big|\Big(r+\frac1{\beta_1}\Big)+\frac1{\beta_1}\Big|K[\t f_5]'(1-\epsilon)-{\beta_1} K[\t f_5](1-\epsilon)\Big| \le C(n^2+1)\epsilon.
  \]
The case when $n=0$ can be easily proved. Similarly, we have $ |\tilde{H_1^n}(r)| \le C(n^2+1)\epsilon$. Next let's prove the estimate for $\tilde{F_1^n}$. Under our assumptions, by \re{deri}, \re{1a}, and \re{2a},
  \[
   \|\t f_7\|_{L^\infty} \le C(n^2+1)+ C\epsilon (n^3+1) ,
  \]
  so that, we can use \re{estK} to derive
  \[
      | K[\t f_7](r) | \le C(n+1) \epsilon,
      \mm | K[\t f_7]'(r) | \le C(n^2+1) \epsilon.
  \]
  The function $\phi = \frac1D  \Big\{K[\t f_7](r) +\Big|K[\t f_7]'(1)\Big|\Big(r+\frac1\beta_2\Big) + \frac1\beta_2\Big|K[\t f_7]'(1-\epsilon)-\beta_2 K[\t f_7](1-\epsilon)\Big|
    + \epsilon\Big\}$ satisfies,
\begin{equation*}
    \begin{split}
        &D\mathscript{L}[\phi] + \frac{\p \phi}{\p r} \frac{\p p_*}{\p r}\\
        =&\;\; \t f_7 + \frac1D  \Big(K[\t f_7]'(r) + \Big|K[\t f_7]'(1)\Big|\Big) \frac{\p p_*}{\p r} -\frac{1}{Dr}\Big|K[\t f_7]'(1)\Big| + \frac{n^2}{D r^2}\Big|K[\t f_7]'(1)\Big|\Big(r+\frac1\beta_2\Big)\\
        &\;\;+ \frac{n^2}{D r^2}\Big(\frac1\beta_2\Big|K[\t f_7]'( 1-\epsilon)-\beta_2 K[\t f_7](1-\epsilon)\Big| +\epsilon\Big)\\
        \ge&\;\; \t f_7 -  C\epsilon \Big\|K[\t f_7]' \Big\|_{L^\infty} +  n^2 \epsilon \; \ge \; \t f_7  -  C  (n^2+1) \epsilon^2 +   n^2\epsilon \; \ge \; \t f_7,
    \end{split}
\end{equation*}    
for $n\ge 1$, where we also make use of \re{deri} in deriving the above estimate. Therefore, it follows from the maximum principle that
\begin{equation*}
    |\tilde{F_1^n}(r)| \le \frac1D  \Big\{K[\t f_7](r) +\Big|K[\t f_7]'(1)\Big|\Big(r+\frac1\beta_2\Big) + \frac1\beta_2\Big|K[\t f_7]'(1-\epsilon)-\beta_2 K[\t f_7](1-\epsilon)\Big|
    + \epsilon\Big\} \le C(n^2+1)\epsilon.
\end{equation*}
Finally, in order to estimate $(p^n_1)'$, we use the explicit formula from Lemma \ref{eqn-n}. Taking $\eta=0$ and $G = (1-n^2)/(1-\epsilon)^2$, we obtain from Lemma \ref{eqn-n} that
\begin{equation}\label{p1nD}
\begin{split}
    (p_1^n)' &= An r^{n-1} - Bn r^{-n-1} + K[f_8]'(r),\\
    \end{split}
\end{equation}
where $A$ and $B$ are defined in Lemma \ref{eqn-n1}. By \re{1a},
\begin{equation*}
    \|f_8\|_{L^\infty} \le C(n^2+1),
\end{equation*}
and together with \re{estK} in Lemma \ref{eqn-n}, we have
\begin{equation}\label{K9est}
    |K[f_8](r)| \le C\frac{n^2+1}{n}\epsilon,\hspace{2em} |K[f_8]'(r)| \le C(n^2+1)\epsilon.
\end{equation}
Combining \re{p1nD} with \re{A} \re{B} and \re{K9est}, we then obtain
\begin{equation*}
    \begin{split}
        \Big|(p_1^n)'\Big| \;\le&\;\;  \max\limits_{1-\epsilon\le r\le1} \Big|\frac{n(r^{n-1}-r^{-n-1})}{(1-\epsilon)^n+(1-\epsilon)^{-n}}\Big[G-k[f_8](1-\epsilon)\Big]\Big| \\
        &\quad\;+ \max\limits_{1-\epsilon\le r\le 1}\Big|\frac{r^{n-1}+(1-\epsilon)^{2n}r^{-n-1}}{1+(1-\epsilon)^{2n}} K[f_8]'(1) \Big| + \max\limits_{1-\epsilon\le r\le 1}\Big|K[f_8]'(r)\Big|\\
        \le&\;\; \Big|\frac{n\,G}{1-\epsilon}\frac{(1-\epsilon)^n-(1-\epsilon)^{-n}}{(1-\epsilon)^n+(1-\epsilon)^{-n}}\Big| + Cn\|K[f_8]\|_{L^\infty} + C\|K[f_8]'\|_{L^\infty}\\
        \le&\;\; \bigg|\frac{n(n^2-1)}{(1-\epsilon)^3}\frac{(1-\epsilon)^n-(1-\epsilon)^{-n}}{(1-\epsilon)^n+(1-\epsilon)^{-n}}\bigg| + C(n^2+1)\epsilon \\
        \le&\;\; 2(n^3 + 1),
    \end{split}
\end{equation*}
hence $ \| ({p_1^n})'\|_{L^\infty(1-\epsilon,1)} \le 2 (n^3+1)$ is valid for sufficiently small $\epsilon$. 
\end{proof}

Based on \re{est1a} and \re{est2a}, the existence and uniqueness of such a solution $(L_1^n, H_1^n, F_1^n, p_1^n)$ to the system \re{L1n1} -- \re{bdyp1n1} can be justified through the contraction mapping principle, hence we have the following lemma.

\begin{lem}
For each nonnegative $n$ and sufficiently small $\epsilon$, the system \re{L1n1} -- \re{bdyp1n1} admits a unique solution $(L_1^n, H_1^n, F_1^n, p_1^n)$.
\end{lem}

By \re{est2a} we already derived the estimate
\begin{equation*}
    \Big|\frac{\p p_1^n(r)}{\p r}\Big| \le 2(n^3 +1), \hspace{2em} 1-\epsilon\le r\le 1.
\end{equation*}
This estimate, however, is not enough; we need a sharper bound for $\frac{\p p_1^n(1-\epsilon)}{\p r}$. To do that, we start with rewriting \re{defL} -- \re{defF} in the same way as in \re{Ls} -- \re{Fs}.

Evaluating \re{r1} at $r=1-\epsilon$, and using \re{Ls} -- \re{Fs}, we obtain
\begin{equation*}
\begin{split}
    \frac{\p^2 L_*(1-\epsilon)}{\p r^2} &= \Big(k_1\frac{(M_0-F_*)L_*}{K_1+L_*}+\rho_1 L_*\Big)\Big|_{r=1-\epsilon} - \frac{1}{1-\epsilon}\frac{\p L_*(1-\epsilon)}{\p r}\\
    &=\rho_3(\gamma+H_0)\Big( \frac{k_1 M_0  }{\lambda K_1+\rho_3(\gamma+H_0)}+\frac{\rho_1 }\lambda  \Big)- \frac{1}{1-\epsilon}\frac{\p L_*(1-\epsilon)}{\p r} + O(\epsilon).
    \end{split}
\end{equation*}
Recall that the boundary condition for $L_*$ is
\begin{equation*}
    \frac{\p L_*(1-\epsilon)}{\p r} = \beta_1(L_*(1-\epsilon)-L_0) = \beta_1\Big(\frac{\rho_3(\gamma+H_0)}{\lambda} + O(\epsilon) - \frac{\rho_3(\gamma+H_0)}{\lambda} - \frac{\epsilon \mu}{\lambda}\Big) = O(\epsilon).
\end{equation*}
We combine the above two equations to derive 
\begin{eqnarray*}
    \frac{1}{\beta_1}\Big(\frac{\p^2 L_*}{\p r^2}-\beta_1\frac{\p L_*}{\p r}\Big)\Big|_{r=1-\epsilon} = \frac{\rho_3(\gamma+H_0)}{\beta_1}\Big( \frac{k_1 M_0  }{\lambda K_1+\rho_3(\gamma+H_0)}+\frac{\rho_1 }\lambda  \Big) + O(\epsilon).
\end{eqnarray*}
Similarly, we can also get
\begin{gather*}
    \frac{1}{\beta_1}\Big(\frac{\p^2 H_*}{\p r^2}-\beta_1\frac{\p H_*}{\p r}\Big)\Big|_{r=1-\epsilon} = \frac{\rho_2 H_0}{\beta_1} + O(\epsilon),\\
    \frac{1}{\beta_2}\Big(\frac{\p^2 F_*}{\p r^2}-\beta_2\frac{\p F_*}{\p r}\Big)\Big|_{r=1-\epsilon} = -\frac{\rho_3(\gamma+H_0)}{\beta_2 D} \;
 \frac{k_1 M_0  }{\lambda K_1+\rho_3(\gamma+H_0)} + O(\epsilon).
\end{gather*}
Comparing with the definitions of $L_*^1$, $H_*^1$ and $F_*^1$ in \re{Ls} -- \re{Fs}, we find that
\begin{gather*}
    \frac{\rho_3(\gamma+H_0)}{\beta_1}\Big( \frac{k_1 M_0  }{\lambda K_1+\rho_3(\gamma+H_0)}+\frac{\rho_1 }\lambda  \Big) = \frac{\mu}{\lambda} - L_*^1,\\
    \frac{\rho_2 H_0}{\beta_1} = -H_*^1, \hspace{2em} -\frac{\rho_3(\gamma+H_0)}{\beta_2 D} \;
 \frac{k_1 M_0  }{\lambda K_1+\rho_3(\gamma+H_0)} = -F_*^1.
\end{gather*}
Therefore, the above equations indicate
\begin{gather}
    \label{bdysL} \frac{1}{\beta_1}\Big(\frac{\p^2 L_*}{\p r^2}-\beta_1\frac{\p L_*}{\p r}\Big)\Big|_{r=1-\epsilon} = \frac{\mu}{\lambda} - L_*^1 + O(\epsilon),\\
    \label{bdysH} \frac{1}{\beta_1}\Big(\frac{\p^2 H_*}{\p r^2}-\beta_1\frac{\p H_*}{\p r}\Big)\Big|_{r=1-\epsilon} = -H_*^1 + O(\epsilon),\\
    \label{bdysF}  \frac{1}{\beta_2}\Big(\frac{\p^2 F_*}{\p r^2}-\beta_2\frac{\p F_*}{\p r}\Big)\Big|_{r=1-\epsilon} = -F_*^1+O(\epsilon).
\end{gather}
After we show \re{bdysL} -- \re{bdysF}, we can combine them with \re{defL} -- \re{defF} as well as \re{est1a} to claim that
\begin{eqnarray}
   \label{expl1n} L_1^n & = & \mu/\lambda - L_*^1 + O((n^2+1) \epsilon),\\
    \label{exph1n} H_1^n & = & -H_*^1 + O((n^2+1) \epsilon),\\
    \label{expf1n} F_1^n & = & -F_*^1 + O((n^2+1) \epsilon).
\end{eqnarray}

With \re{expl1n} -- \re{expf1n}, we are able to derive a more delicate estimate for $\frac{\p p_1^n(1-\epsilon)}{\p r}$. Substituting \re{Ls} -- \re{Fs} and \re{expl1n} -- \re{expf1n} all into \re{f9}, recalling also \re{mum} and \re{rho4}, we obtain
\begin{equation}
\label{f9e} f_8 = \frac{\mu}{\gamma+H_0} - \frac{1}{M_0}\Big(\frac{M_0(\lambda L_*^1-\rho_3 H_*^1)}{\gamma+H_0} - \rho_4 F_*^1\Big) + O((n^2+1) \epsilon) = \frac{\mu}{\gamma+H_0} + O((n^2+1) \epsilon),
\end{equation}
and we are ready to establish the following lemma.
\begin{lem}
  \label{sharpest}
For each nonnegative $n$ and small $0<\epsilon\ll1$, the following inequality holds:
 \begin{equation}
      \label{sharp}
      \bigg|\frac{\p p_1^n(1-\epsilon)}{\p r} -\frac{\epsilon \mu}{\gamma+H_0} -\frac{n[(1-\epsilon)^{2n}-1]}{(1-\epsilon)[(1-\epsilon)^{2n}+1]}G \bigg| \le C (n^2+1) \epsilon^2,
 \end{equation}
where $G=(1-n^2)/(1-\epsilon)^2$, and the constant $C$ is independent of $\epsilon$ and $n$.
\end{lem}
\begin{proof}
The estimate \re{sharp} shall be established  by using the explicit formula from Lemma \ref{eqn-n}. Specifically, we take $\eta = \frac{\mu}{\gamma+H_0}$ and $f(r) = f_8 - \eta $. From \re{f9e}, we have
\begin{equation*}
    \|f\|_{L^\infty} = \|f_8 - \eta\|_{L^\infty} \le C(n^2 + 1) \epsilon;
\end{equation*}
we then combine it with Lemma \ref{eqn-n} to derive
\begin{equation}\label{estkf}
    | K[f](r) | \le  C(n+1)\epsilon^2, \mm
  | K[f]'(r) | \le  C(n^2+1)\epsilon^2.
\end{equation}
Following Lemmas \ref{eqn-n} and   \ref{eqn-n1}, we can explicitly solve $p_1^n$  as 
\begin{equation*}
    p_1^n(r) = \psi_1(r) + A r^n + B r^{-n} + K[f](r).
\end{equation*}
For $\psi_1(r)$, we use \re{psi} and Lemma \ref{nepsilon} to obtain
\begin{equation*}
    \psi_1(1-\epsilon)  =\frac{\eta}{n(n+2)} + O\Big(\frac{\epsilon^2}{n}\Big), \mm 
    \psi_1'(1-\epsilon) =  \frac{2\eta\epsilon}{n+2} + O(\epsilon^2), \mm n\neq 0,
\end{equation*}
and
\be \label{psi10}
 \psi_1(1-\epsilon) = O(\epsilon^2), \m \psi_1'(1-\epsilon) =
 {\eta\epsilon} +O(\epsilon^2), \m n=0.
\ee
Together with \re{A} \re{B} as well as \re{estkf} the first derivative of $p_1^n$ at $r=1-\epsilon$ evaluates to
\begin{equation*}
\begin{split}
    \frac{\p p_1^n(1-\epsilon)}{\p r} &=\; \psi_1'(1-\epsilon) + An (1-\epsilon)^{n-1} - Bn(1-\epsilon)^{-n-1} + K[f]'(1-\epsilon)\\
    &=\;\psi_1'(1-\epsilon) + \frac{n[(1-\epsilon)^{2n}-1]}{(1-\epsilon)[(1-\epsilon)^{2n}+1]}\Big(G-\psi_1(1-\epsilon)\Big) + O((n^2+1)\epsilon^2)\\
    &=\;\eta \epsilon + \frac{n[(1-\epsilon)^{2n}-1]}{(1-\epsilon)[(1-\epsilon)^{2n}+1]}G + O((n^2+1)\epsilon^2), \mm n\neq 0,
    \end{split}
\end{equation*}
which is equivalent to \re{sharp}. It is clear from \re{psi10} that
the above formula is also valid for $n=0$.
\end{proof}

Like in \re{J1}, we denote
\begin{equation}
    \label{J2}
    J_2^n(\mu,\rho_4) = \frac{1}{\epsilon^2}\bigg[\frac{\p p_1^n(1-\epsilon)}{\p r}-\frac{\epsilon \mu}{\gamma+H_0} -\frac{n(1-n^2)[(1-\epsilon)^{2n}-1]}{(1-\epsilon)^3[(1-\epsilon)^{2n}+1]}\bigg],
\end{equation}
which indicates
\begin{equation}
    \label{J22}
    \frac{\p p_1^n(1-\epsilon)}{\p r} = \frac{\epsilon \mu}{\gamma+H_0}+ \frac{n(1-n^2)[(1-\epsilon)^{2n}-1]}{(1-\epsilon)^3[(1-\epsilon)^{2n}+1]}+\epsilon^2 J_2^n(\mu,\rho_4).
\end{equation}
From Lemma \ref{sharpest}, we immediately obtain that there exists a constant $C$ which is independent of $n$ and $\epsilon$ such that
\begin{equation*}
    |J_2^n(\mu,\rho_4)|\le C(n^2+1).
\end{equation*}
In addition, we also need to estimate $\frac{\dif J_2^n}{\dif \mu}$. To do that, we take $\mu$ derivative of equation \re{J22} to obtain
\begin{equation}
\label{J22D}
    \frac{\dif J_2^n}{\dif \mu} = \frac{\p J_2^n}{\p \mu} + \frac{\p J_2^n}{\p \rho_4 }\frac{\p \rho_4}{\p \mu} = \frac{1}{\epsilon^2} \Big[\frac{\p}{\p r}\Big(\frac{\p p_1^n}{\p \mu}\Big)\Big|_{r=1-\epsilon} - \frac{\epsilon}{\gamma+H_0}\Big].
\end{equation}
In order to estimate the right-hand side of \re{J22D}, we differentiate the whole system \re{L1n1} -- \re{bdyp1n1} in $\mu$ and follow the same procedures as in Lemmas \ref{lem4.1a} and   \ref{sharpest}. Consequently, a similar result as \re{sharp} can be obtained, i.e.,
\begin{equation*}
    \bigg|\frac{\p}{\p r}\Big(\frac{\p p_1^n}{\p \mu}\Big)\Big|_{r=1-\epsilon} - \frac{\epsilon}{\gamma+H_0}\bigg| \le C(n^2+1)\epsilon^2.
\end{equation*}
Combined with \re{J22D}, it follows that $\Big| \frac{\dif J_2^n}{\dif \mu}\Big| \le C(n^2+1)$. Therefore we have the following lemma. 

\begin{lem}\label{J2prop} 
For function $J_2^n(\mu,\rho_4)$ defined in \re{J2}, there exists a constant $C$ which is independent of $\epsilon$ and $n$ such that
\begin{equation}
    \label{J2bounds}
    |J_2^n(\mu,\rho_4(\mu))| \le C(n^2+1),\hspace{2em} \Big|\frac{\dif J_2^n(\mu, \rho_4(\mu))}{\dif \mu}\Big| \le C(n^2+1).
\end{equation}
\end{lem}

At this point, we are finally ready to prove our main result Theorem \ref{result}. 

\begin{proof}[\textbf{Proof of Theorem \ref{result}}]
Substituting \re{nt2} into \re{FreD}, we obtain the Fr\'echet derivative of $\mathcal{F}(\tilde{R},\mu)$ in $\tilde{R}$ at the point $(0,\mu)$, namely,
  \begin{equation*}
      [\mathcal{F}_{\tilde{R}}(0,\mu)]\cos(n\theta) = \Big(\frac{\p^2 p_*(1-\epsilon)}{\p r^2} + \frac{\p p_1^n(1-\epsilon)}{\p r}\Big)\cos(n\theta);
  \end{equation*}
we then combine the above formula with \re{J1} and \re{J22} to derive
\begin{equation}\label{Fre}
      [\mathcal{F}_{\tilde{R}}(0,\mu)]\cos(n\theta) = \Big(
      \frac{\epsilon \mu}{\gamma+H_0} + \frac{n(1-n^2)[(1-\epsilon)^{2n}-1]}{(1-\epsilon)^3[(1-\epsilon)^{2n}+1]} + \epsilon^2(J_1+J_2^n)\Big)\cos(n\theta).
\end{equation}
For fixed nonnegative $n$, 
\begin{equation*}
    \frac{(1-\epsilon)^{2n}-1}{(1-\epsilon)[(1-\epsilon)^{2n}+1]} = -n\epsilon + O(n^2 \epsilon^2),
\end{equation*}
when $\epsilon$ is sufficiently small so that $n\epsilon < 1$. In this case, the equation $[\mathcal{F}_{\tilde{R}}(0,\mu)]\cos(n\theta)=0$ is satisfied if and only if
\begin{equation}\label{U}
    U(\mu, \epsilon) \triangleq \frac{\mu}{\gamma+H_0} - n^2(1-n^2) + \epsilon (J_1 + J_2^n) + O(n^5 \epsilon) = 0.
\end{equation}
Notice that both $J_1$ and $J_2^n$ contain $\mu$, it is impossible to solve $\mu$ explicitly from equation \re{U}. However, we are able to claim that for each small $\epsilon$, \re{U} admits a unique solution $\mu$. To prove it, we first find that $U((\gamma+H_0)n^2(1-n^2),0) = 0$; in addition, if we take partial $\mu$ derivative on both sides of \re{U} and evaluate the value at $(\mu,0)$, we have
\begin{equation*}
    \frac{\p}{\p \mu}U(\mu,0) = \Big[\frac{1}{\gamma+H_0} + \epsilon \Big(\frac{\dif J_1}{\dif \mu} + \frac{\dif J_2^n}{\dif \mu}\Big)\Big]\Big|_{\epsilon=0}=\frac{1}{\gamma+H_0} > 0.
\end{equation*}
Therefore, it follows from the implicit function theorem that, for each small $\epsilon$, there exists a unique solution, which is close to $(\gamma+H_0)n^2(1-n^2)$, such that equation \re{U} is satisfied; we denote the unique solution by $\mu_n$. In what follows, we shall justify that $\mu=\mu_n$ with $n\ge 2$ and $\mu_n>\mu_c$ is a bifurcation point for the system \re{e1} -- \re{e8} when $\epsilon$ is sufficiently small. 

What we need to do is to verify the four assumptions of the Crandall-Rabinowitz theorem (Theorem \ref{CRthm}) at the point $\mu=\mu_n$. To begin with, the assumption (1) is naturally satisfied due to Theorem \ref{thm21}. To be more specific, for each $\mu_n > \mu_c$, we can find a small $\epsilon^*>0$, such that for $0<\epsilon<\epsilon^*$, there exists a unique radially symmetric stationary solution, hence $\mathcal{F}(0,\mu_n)=0$. Next let's proceed to verify the assumption (2) and (3) for a fixed small $\epsilon$. It suffices to show that for every $m$, $m\neq n$,
\begin{equation}
    \label{mneqn} [\mathcal{F}_{\tilde{R}}(0,\mu_n)]\cos(m\theta) \neq 0, \mm m\neq n,
\end{equation}
or equivalently,
\begin{equation}
    \label{mneqn1} W(m)\triangleq  \frac{\epsilon \mu_n}{\gamma+H_0} + \frac{m(m^2-1)[1-(1-\epsilon)^{2m}]}{(1-\epsilon)^3[(1-\epsilon)^{2m}+1]} + \epsilon^2\Big(J_1(\mu_n,\rho_4)+J_2^m(\mu_n,\rho_4)\Big) \neq 0, \m m\neq n.
\end{equation}
To establish statement \re{mneqn} (or statement \re{mneqn1}), we 
split the proof into three cases:

\vspace{5pt}
\noindent \underline{Case (i)} $m>\max\{2n,m_0\}$ and $m\epsilon \le \frac12$, where $m_0$ will be determined later. Using the inequality $m\epsilon \le \frac12$, together with Lemma \ref{nepsilon}, we deduce that
\begin{equation*}
    (1-\epsilon)^{2m} \le 1- 2m\epsilon + 2m^2\epsilon^2 \le 1 - 2m\epsilon + m\epsilon \le 1-m\epsilon,
\end{equation*}
hence (recall that $n\ge 2$ so that $m > 4$ in this case)
\begin{equation}\label{mest}
    \frac{m(m^2-1)[1-(1-\epsilon)^{2m}]}{(1-\epsilon)^3 [(1-\epsilon)^{2m}+1]} \ge \frac\epsilon 2 m^2(m^2-1).
\end{equation}
In addition, by Lemma \ref{J1prop} and Lemma \ref{J2prop}, there exists a constant $C$ which does not depend on $\epsilon$ and $m$ such that,
\begin{equation}\label{J12}
    |J_1 + J_2^m| \le |J_1| + |J_2^m| \le C(m^2+1).
\end{equation}
Substituting \re{mest} and \re{J12} into \re{mneqn1}, we derive
\begin{equation*}
    W(m) \ge \frac{\epsilon \mu_n}{\gamma+H_0} + \frac\epsilon 2 m^2(m^2-1) - C\epsilon^2 (m^2+1) = \epsilon \Big[ \frac{\mu_n}{\gamma+H_0} + \frac12 m^2(m^2-1) - C\epsilon (m^2+1)\Big].
\end{equation*}
It is clear that $W(m)>0$ for large $m$ as the leading order term in the brackets is $\frac{m^4}{2}$; hence we can find $m_0>0$ such that when $m> m_0$,
    \begin{equation*}
        W(m) >  0.
    \end{equation*}

\noindent \underline{Case (ii)} $m>\max\{2n,m_0\}$ and $m\epsilon > \frac12$. In this case, we have
    \begin{equation*}
        (1-\epsilon)^{2m} \le \Big(1-\frac{1}{2m}\Big)^{2m} \le e^{-1},
    \end{equation*}
and hence (since $m>\max\{2n,m_0\}$, we also have $m > 4$ in this case)
\begin{equation*}
    \frac{m(m^2-1)[1-(1-\epsilon)^{2m}]}{(1-\epsilon)^3 [(1-\epsilon)^{2m}+1]} \ge \frac{1-e^{-1}}{2} m(m^2-1).
\end{equation*}
Similar as in Case (i), we substitute the above inequality as well as \re{J12} into \re{mneqn1}, and derive
    \begin{equation*}
        W(m) \ge \frac{\epsilon \mu_n}{\gamma+H_0} + \frac{1-e^{-1}}{2} m(m^2-1)  - C\epsilon^2 (m^2+1);
    \end{equation*}
notice that the leading order term is $\frac{1-e^{-1}}{2}m^3$, we can easily find a bound for $\epsilon$, denoted by $E_1$, such that when $\epsilon<E_1$,
    \begin{equation*}
        W(m) \ge \frac{1-e^{-1}}{4} m(m^2-1) > 0.
    \end{equation*}

\noindent \underline{Case (iii)} $0\le m\le \max\{2n,m_0\}$. From our previous analysis, we know that $\mu_n$ is close to $(\gamma+H_0)n^2(1-n^2)$.
Since $\max\{2n,m_0\}$ is a finite number, 
we can choose $\epsilon$ small and similarly
define all $\mu_m$ for $m\le \max\{2n,m_0\}$
so that $\mu_m$ is close to $(\gamma+H_0)m^2(1-m^2)$; 
in this case $W(m)\neq 0$
if and only if $\mu_n\neq \mu_m$.
To be more specific, we have
\begin{equation*}
    \lim\limits_{\epsilon\rightarrow 0}\mu_n = (\gamma+H_0)n^2(1-n^2),\hspace{2em} \lim\limits_{\epsilon\rightarrow 0}\mu_m = (\gamma+H_0)m^2(1-m^2).
\end{equation*}
Since $m\neq n$ and $n\ge 2$, it follows that
\begin{equation}\label{lim}\begin{split}
    \lim\limits_{\epsilon\rightarrow 0} |\mu_n - \mu_m| &\ge \min\{ \lim\limits_{\epsilon\rightarrow 0} |\mu_n - \mu_{n-1}|,  \lim\limits_{\epsilon\rightarrow 0} |\mu_n - \mu_{n+1}| \} \\
    &= (\gamma+H_0)(4n^3-6n^2+2n) \ge 12(\gamma+H_0).\end{split}
\end{equation}
Recall again  $m$ is bounded in this case, we can find a bound for $\epsilon$, denoted by $E_2$, such that when $\epsilon < E_2$,
\begin{equation*}
    \Big|\mu_n - \lim\limits_{\epsilon\rightarrow 0}\mu_n\Big| + \max\limits_{0\le m\le \max\{2n,m_0\}}\Big|\mu_m - \lim\limits_{\epsilon\rightarrow 0}\mu_m\Big| \le 6(\gamma+H_0),
\end{equation*}
together with \re{lim}, we obtain
\begin{equation*}
    \begin{split}
        \Big|\mu_n - \mu_m\Big| &\ge \Big|\lim\limits_{\epsilon\rightarrow 0}\mu_n-\lim\limits_{\epsilon\rightarrow 0}\mu_m\Big| - \Big|\mu_n - \lim\limits_{\epsilon\rightarrow 0}\mu_n\Big| - \Big|\mu_m-\lim\limits_{\epsilon\rightarrow 0}\mu_m\Big| \ge 6(\gamma+H_0) > 0.
    \end{split}
\end{equation*}

\vspace{5pt}

By combining all three cases, the assumptions (2) and (3) in Theorem \ref{CRthm} are satisfied when $\epsilon$ is sufficiently small, i.e., 
\beaa 
 && \text{Ker} \,\mathcal{F}_{\tilde{R}}(0,\mu_n) = \text{span}\{\cos(n\theta)\},\\
 && Y_1 = \text{span}\{1,\cos(\theta),\cdots,\cos((n-1)\theta),\cos((n+1)\theta),\cdots\}, \\
 && \text{and } Y_1 \mbox{$\bigoplus$} \text{Ker} = X^{l+\alpha}_1,
\eeaa
such that the spaces listed here (codimension space, non-tangential space) meet the requirements of the Crandall-Rabinowitz Theorem. To finish the whole proof, it remains to show the last assumption. Differentiating \re{Fre} in $\mu$, we derive
\begin{equation}\label{assump4}
    \begin{split}
     \left[\mathcal{F}_{\mu \tilde{R}}(0,\mu)\right] \cos(n\theta) &= \Big(
      \frac{\epsilon}{\gamma+H_0} + \epsilon^2\Big(\frac{\dif J_1}{\dif \mu} + \frac{\dif J_2^n }{\dif \mu}\Big)\Big)\cos(n\theta)\\
      &= \epsilon \Big(\frac{1}{\gamma+H_0} + \epsilon\Big(\frac{\dif J_1}{\dif \mu} + \frac{\dif J_2^n }{\dif \mu}\Big)\Big)\cos(n\theta).
      \end{split}
\end{equation}
By Lemma \ref{J1prop} and Lemma \ref{J2prop}, there exists a constant $C$ independent of $\epsilon$ and $n$ such that 
\begin{equation}\label{E3}
    \Big|\frac{\dif J_1}{\dif \mu} + \frac{\dif J_2^n }{\dif \mu}\Big| \le \Big|\frac{\dif J_1}{\dif \mu}\Big| +\Big| \frac{\dif J_2^n }{\dif \mu}\Big|\le C(n^2+1).
\end{equation}
Based on \re{E3}, we can find a bound $E_3$ (depending on $n$), such that when $\epsilon < E_3$, 
\begin{equation*}
    \frac{1}{\gamma+H_0} + \epsilon\Big(\frac{\dif J_1}{\dif \mu} + \frac{\dif J_2^n }{\dif \mu}\Big) >  \frac{1}{\gamma+H_0} - C E_3 (n^2+1)  > 0,
\end{equation*}
and hence $  \left[\mathcal{F}_{\mu \tilde{R}}(0,\mu)\right] \cos(n\theta) \not\in Y_1$,
i.e., the assumption (4) is satisfied.

Combining all pieces together, we take $E=\min\{\epsilon^*, E_1,E_2,E_3\}$, then we know that all the assumptions of the Crandall-Rabinowitz theorem are satisfied when $\epsilon\in (0,E)$. Hence we conclude that $\mu=\mu_n$ is a symmetry-breaking bifurcation point.
\end{proof}

\section{Appendix}
\subsection{A supersolution}\label{app-super} As in \cite{FHHplaque}, we use the function $$\xi(r)=\frac{1-r^2}{4}+\frac12 \log r$$
a lot when we apply the maximum principle. Recall that $\xi$ satisfies
\begin{eqnarray*}
        &&-\Delta \xi = 1, \hspace{1em} \xi_r(r) = \frac{1-r^2}{2r}, \hspace{1em} \text{and } \hspace{1em} \xi(r)=O(\epsilon^2) \text{ when } 1-\epsilon < r < 1.
\end{eqnarray*}
Next we take 
$$c_1(\beta_1,\epsilon)= \frac{1}{\beta_1}\frac{\epsilon(2-\epsilon)}{2(1-\epsilon)}-\frac{\epsilon(2-\epsilon)}{4}-\frac12 \log(1-\epsilon)\equiv \frac{\epsilon}{\beta_1} + O(\epsilon^2),\hspace{1em}\text{and}\hspace{1em} c_2(\beta_1,\tau) = \frac{2}{\beta_1} |\tau|.$$
It is easy to verify that
\begin{equation}
    \Big[-\frac{\p \xi}{\p r} + \beta_1\Big(\xi+c_1(\beta_1,\epsilon)\Big)\Big]\Big|_{r=1-\epsilon}=\Big[-\frac{\p \Big(\xi+c_1(\beta_1,\epsilon)\Big)}{\p r} + \beta_1\Big(\xi+c_1(\beta_1,\epsilon)\Big)\Big]\Big|_{r=1-\epsilon}=0. \label{lem31-4}
\end{equation}
Using \re{lem31-4}, also recalling $\|S(\theta)\|_{C^{4+\alpha}(\Sigma)}\le 1$, we derive
\begin{equation*}
    \begin{split}
        &\bigg[\frac{\p \Big(\xi+c_1(\beta_1,\epsilon) + c_2(\beta_1, \tau)\Big)}{\p {\bm n}}+\beta_1\Big(\xi+c_1(\beta_1,\epsilon) + c_2(\beta_1, \tau)\Big)\bigg]\bigg|_{r=1-\epsilon+\tau S} \\
        =&\Big[-\frac{\p \xi}{\p r}\frac{1}{\sqrt{1+(\tau S')^2}} + \beta_1 \xi\Big]\Big|_{r=1-\epsilon+\tau S} + \beta_1 c_1(\beta_1,\epsilon) + \beta_1 c_2(\beta_1,\tau)\\
        =&\Big[-\frac{\p \xi}{\p r}+\beta_1 \Big(\xi+c_1(\beta_1,\epsilon)\Big)\Big]\Big|_{r=1-\epsilon} +\Big[-\frac{\p^2 \xi}{\p r^2}+\beta_1 \frac{\p \xi}{\p r}\Big]\Big|_{r=1-\epsilon}\tau S + 2|\tau| + O(|\tau S|^2) + O(|\tau S'|^2)\\
        =&0 + \Big[\frac{1+(1-\epsilon)^2}{2(1-\epsilon)^2} + \beta_1 \frac{1-(1-\epsilon)^2}{2(1-\epsilon)}\Big]\tau S + 2|\tau| +  O(|\tau S|^2) + O(|\tau S'|^2)\\
        =&(1+O(\epsilon))\tau S + 2|\tau| + O(|\tau S|^2) + O(|\tau S'|^2) > 0.
    \end{split}
\end{equation*}

\subsection{Transformation $T_\tau$}\label{app-2}
The transformation $T_\tau$
\begin{equation*}
    T_\tau: \tilde{r}=\frac{r-1}{2(\epsilon-\tau S(\theta))}+1, \quad \tilde{\theta}=\frac{\theta}{\epsilon}
\end{equation*}
maps $\Omega_\tau$ to a long stripe region $(\tilde{r},\tilde{\theta})\in [\frac12,1]\times[0,\frac{2\pi}{\epsilon}]$. Let $y$ satisfies
\begin{eqnarray}
    &-\Delta y =  -\frac1r\frac{\p}{\p r} \Big(r\frac{\p y}{\p r}\Big)-\frac1r\frac{\p }{\p \theta}\Big(\frac1r \frac{\p y}{\p \theta} \Big)= f(y) \hspace{2em} &\text{in } \Omega_\tau,
\end{eqnarray}
and set $\tilde{y}(\tilde{r},\tilde{\theta})=y(r,\theta)-y_0$. Obviously, $\tilde{y}$ should be $\frac{2\pi}{\epsilon}$-periodic in $\tilde{\theta}$. Using the chain rule, we obtain:
\begin{equation*}
    \begin{split}
        &\frac{\p }{\p r}=\frac{\p }{\p \tilde{r}}\frac{\p \tilde{r}}{\p r} = \frac{1}{2(\epsilon-\tau S)}\frac{\p }{\p \tilde{r}} = \Big(\frac{1}{2\epsilon} + O(\tau S)\Big)\frac{\p}{\p \tilde{r}}\, ,\\
        &\frac{\p }{\p \theta}=\frac{\p }{\p \tilde{r}}\frac{\p \tilde{r}}{\p \theta}+\frac{\p }{\p \tilde{\theta}}\frac{\p \tilde{\theta}}{\p \theta}=\frac{\tau S' }{\epsilon-\tau S}\Big(\tilde{r}-1\Big)\frac{\p }{\p \tilde{r}}  + \frac{1}{\epsilon} \frac{\p }{\p \tilde{\theta}}=O(\tau \|S\|_{C^1})\frac{\p}{\p \tilde{r}} + \frac{1}{\epsilon}\frac{\p }{\p \tilde{\theta}}\, .
        \end{split}
\end{equation*}
Hence we can write the equation of  $\tilde{y}(\tilde{r},\tilde{\theta})$ as
\begin{equation*}
     -\frac{\p }{\p \tilde{r}}\Big((1+ A_1) \frac{\p \tilde{y}}{\p \tilde{r}} + A_2 \frac{\p \tilde{y}}{\p \tilde{\theta}}\Big) - \frac{\p}{\p \tilde{\theta}} \Big(A_3\frac{\p \tilde{y}}{\p \tilde{r}} + (1+A_4) \frac{\p \tilde{y}}{\p \tilde{\theta}}\Big) +  A_5\frac{\p \tilde{y}}{\p \tilde{r}} + A_6\frac{\p \tilde{y}}{\p \tilde{\theta}} = \epsilon^2 \tilde{f}(\tilde{y}),
\end{equation*}
where $\tilde{f}(\tilde{y}) = r f(y)$, and $A_1,A_2,A_3,A_4,A_5,A_6 \sim O(\epsilon)$ are thus bounded. Furthermore, since $A_1,A_2,A_3,A_4$ contain at most first derivative of $S$, they are $C^\alpha$ if $S\in C^{1+\alpha}$.

\subsection{A continuation lemma}\label{appcon}
The next lemma concerns the continuation of estimates. The proof is standard and we omit the details.
\begin{lem} \label{con}
  Let $\{ \vec{Q}_\delta^{(i)}\}_{i=1}^M$ be a finite collection of real vectors, and define the norm of the vector by $|\vec{Q}_\delta|_{\max}=\max\limits_{1\le i \le M}|Q_\delta^{(i)}|$. Suppose that $0<C_1<C_2$, and
  
  (i) $|\vec{Q}_0|_{\max}\le C_1$;
  
  (ii) For any $0< \delta \le 1$, if $|\vec{Q}_\delta|_{\max}\le C_2$, then $|\vec{Q}_\delta|_{\max}\le C_1$;
  
  (iii) $\vec{Q}_\delta$ is continuous in $\delta$.
  
\noindent Then $|\vec{Q}_\delta|_{\max}\le C_1$ for all $0< \delta \le 1$.
\end{lem}

\begin{rem}
If the finite collection is replaced by an infinite collection, then
(iii) will need to be replaced by ``uniform continuity" in $\delta$.
\end{rem}

\end{document}